\documentclass[a4paper]{amsart}
\usepackage{amssymb}
\usepackage{verbatim}
\usepackage{amscd}
\def\card{{{\operatorname{card}}}}

\numberwithin{equation}{section}
\theoremstyle{plain}
\newtheorem{theorem}[equation]{Theorem}
\newtheorem{corollary}[equation]{Corollary}

\newtheorem{lemma}[equation]{Lemma}
\newtheorem*{(DQ1)}{(DQ1)}
\theoremstyle{definition}

\theoremstyle{remark}

\begin{document}
\title [Families of directed graphs]{Families of directed graphs  and topological conjugacy of the associated Markov-Dyck shifts }
\author{Toshihiro Hamachi}
\author{Wolfgang Krieger}
\begin{abstract}
We describe structural properties of strongly connected finite directed graphs, that are invariants of the topological conjugacy of their Markov-Dyck shifts.  For strongly connected finite directed graphs with these properties topological conjugacy of their Markov-Dyck shifts implies isomorphism of the graphs. 
\end{abstract}
\maketitle

\section{Introduction}

Let $\Sigma$ be a finite alphabet, and let $S $ be the left shift on
$\Sigma^{\Bbb Z}$,
$$
S((x_i)_{ i \in {\Bbb Z}})   = (x_{i+1})_{ i \in {\Bbb Z}}, \qquad  
(x_i)_{ i \in {\Bbb Z}} \in \Sigma^{\Bbb Z}.
$$
The closed shift-invariant subsystems of the shifts $S$ are called
subshifts. For an introduction to the theory of subshifts see \cite{Ki} and \cite{LM}. 
A finite word in the symbols of 
$\Sigma$  is called admissible for the subshift 
$X \subset \Sigma^{\Bbb Z} $ if it
appears somewhere in a point of $X$. 
A subshift $X \subset \Sigma^{\Bbb Z} $ is uniquely determined by its language 
of admissible words. 

In this paper we continue the study \cite {KM} of  the topological conjugacy of Markov-Dyck shifts. The Markov-Dyck shift of a strongly connected finite directed graph is constructed via the graph inverse semigroup. We denote a finite directed graph $G$ with
vertex set ${\mathcal V}$ and edge set 
${\mathcal E}$ by $G(\mathcal V  , \mathcal E  )$. 
The source vertex of an edge $e \in {\mathcal E}$  
we denote by
$s$ and its target vertex by $t$.
Given a finite directed graphs $G = G(\mathcal V  , \mathcal E)$,
let 
$
{\mathcal E}^- = \{e^-: e \in   \mathcal E  \}
$
be a copy of ${\mathcal E}$.
Reverse the directions of the edges in ${\mathcal E}^-$
to obtain the 
edge set
$
{\mathcal E}^+ = \{e^+: e \in   \mathcal E  \}
$
of the reversed graph 
of $G(\mathcal V  , \mathcal E^-)  $.
In this way one has defined a directed graph 
${G}( \mathcal V  , {\mathcal E}^- \cup \thinspace \mathcal E^+  )$, that has the directed graphs
$G(\mathcal V, \mathcal E^-)$ and   $G(\mathcal V, \mathcal E^+)$ as subgraphs.
With idempotents $\bold 1_V, V \in {\mathcal V},$ 
the set 
${\mathcal E}^- \cup \{\bold 1_V:V \in {\mathcal V}\}\cup {\mathcal E}^+$
is the generating set of the graph inverse semigroup $\mathcal S(G)$ of $G$ (see \cite [Section 10.7]{L}), where, besides
$\bold 1_V^2 = \bold 1_V, V \in {\mathcal V},$ the relations are
$$
\bold 1_U\bold 1_W = 0, \qquad   U, W \in {\mathcal V}, U \neq W,  
$$
$$
f^-g^+ =
\begin{cases}
\bold 1_{s(f)}, &\text{if  $f = g$}, \\
0, &\text {if  $f \neq g$},\quad f,g \in {\mathcal E},
\end{cases}
$$
\begin{equation*}
\bold 1_{s(f)} f^- = f^- \bold 1_{t(f)},\qquad
\bold 1_{t(f)} f^+ = f^+ \bold 1_{s(f)},\qquad
f \in {\mathcal E}. 
\end{equation*}
The directed graphs with a single vertex and $N>1$
loops yield the Dyck inverse monoids (the "polycycliques" of \cite{NP}), that we denote by 
$\mathcal D_N$. 
 
We consider strongly connected finite directed graphs $G =G(\mathcal V, \mathcal E)$,
that we assume not to be a cycle.
From the graph $G$ one obtains a Markov-Dyck shift
$M\negthinspace {\scriptstyle D}({G})$, that has as alphabet the set 
${\mathcal E}^-\cup {\mathcal E^+}$,
and a word 
$(e_k)_{1 \leq k \leq K}$ in the symbols of ${\mathcal E}^-\cup {\mathcal E^+}$
is admissible for 
$M\negthinspace {\scriptstyle D}({G})$
precisely if 
$$
\prod_{1 \leq k \leq K}e_k \neq 0.
$$
The directed graphs with a single vertex and $N>1$
loops yield the Dyck shifts $D_N$ \cite {Kr1}.

For a directed graph $G( \mathcal V , \mathcal E )$ we  denote by $\mathcal R_G$ the set of vertices of $G$, that have at least two incoming edges, and
we denote  by $\mathcal F_G$ the set of edges that are the only incoming edges of their target  vertices. 
The graph  $G(\mathcal V,\mathcal F_G)$ is a directed subgraph of $G$, that decomposes into directed trees, that we refer to as the subtrees of $G$. 
The set  
$\mathcal R_G$
 is the set of  roots of the subtrees. A subtree is  equal to the one-vertex tree with vertex $V \in \mathcal R_G$, if $V$ is not the source vertex of an edge in $\mathcal F_G$.
Contracting the subtrees of $G$, that are not one-vertex, to their roots yields a directed graph, that we denote by $\widehat G$.  In \cite{Kr2} a Property (A) of subshifts, that is an invariant of topological conjugacy, was introduced and a semigroup, that is invariantly associated with  a subshift with Property (A), was constructed.
That the Markov-Dyck shifts have Property (A) was shown in \cite [Section 2] {HK}.
 The semigroup $\mathcal S(M\negthinspace {\scriptstyle D}({G}))$ that is associate to  
$M\negthinspace {\scriptstyle D}({G})$ 
is 
$\mathcal S(\widehat G)$ \cite [Section 3] {HK}, 
which implies, that the isomorphism class of the graph  $\widehat G$ is an invariant of topological conjugacy of the Markov-Dyck shift of $G$ \cite[Corollary 3.2]{Kr4},
\cite[Theorem 2.1]{Kr5}.

In \cite{KM}  
three families, $\bold{F}_I,\bold{F}_{II}$ and $\bold{F}_{III}$ of directed graphs 
$G =G( \mathcal V , \mathcal E )$, such that $G(\mathcal V, \mathcal F_G)$ is a tree,
were considered. The Markov-Dyck shifts of the graphs in each of these families  were characterized among the Markov-Dyck shifts by invariants of topological conjugacy. It was shown for the graphs in each of these families, that the topological conjugacy class of their Markov-Dyck shifts determines the isomorphism class of the graphs.
The family $\bold{F}_I$ contains the graphs $G =G( \mathcal V , \mathcal E )$, such that
$G(\mathcal V, \mathcal F_G)$ is a tree, and 
such that all of its vertices, except the root of the subtree, have out-degree one.
The family $\bold{F}_{II} $ contains the graphs $G =G( \mathcal V , \mathcal E )$, such that $G(\mathcal V, \mathcal F_G)$ is a tree
 such that  all leaves of the subtree 
 are at level one.
The family $\bold{F}_{III}$  contains the the graphs $G =G( \mathcal V , \mathcal E )$, such that 
$G(\mathcal V, \mathcal F_G)$ is a tree, that has the shape of a "V",
and that are such that the two leaves of the  subtree 
have the same out-degree in $G$, 
and all interior vertices of the  subtree 
have out-degree one in $G$.
 
This paper is a sequel of \cite{KM} that continues this line of investigation.
 We introduce two families  of  graphs 
$G(\mathcal V, \mathcal E)$, such that    
$G(\mathcal V, \mathcal F_G)$ is a tree, that we name $\bold{F}_{IV}$ and $\bold{F}_{V}$.
We characterize the Markov-Dyck shifts of the graphs in  these families among the Markov-Dyck shifts by invariants of topological conjugacy, and we show, that the topological conjugacy class of their Markov-Dyck shifts in these families determines the isomorphism class of the graphs.

In  Section 3 we consider  the family 
$\bold{F}_{IV}$, that we define as the family of graphs $G = G(\mathcal V, \mathcal E)$, such that     
$G(\mathcal V, \mathcal F_G)$ is a tree, such that
$
\card(\mathcal E\setminus \mathcal F_G) = 4,
$
and such that 
there is an $H \in \Bbb N$, such that  the leaves of the subtree  
are at level $H$,
and such that one finds two branch points, each of out-degree two, on the directed  cycles of minimal length, that pass through the root of the  subtree. 
It is also assumed, that the tree $G(\mathcal V, \mathcal F_G)$ is of small size. 

 In Section 4 we consider the family $\bold{F}_{V}$
of graphs $G = G(\mathcal V, \mathcal E)$ such that    
$G(\mathcal V, \mathcal F_G)$ is a tree,
that   
satisfy a strong structural condition: The allowed subtrees are line graphs, or they can be obtained by replacing in a caterpillar tree the "legs" of the caterpillar by line graphs. 

A rooted tree is called spherically homogeneous, if all of its leaves are at the same level, and if all  vertices at the same level have the same out-degree. We say that a graph
$G = G(\mathcal V, \mathcal E)$ such that    
$G(\mathcal V, \mathcal F_G)$ is a tree,
is spherically homogeneous if 
its subtree
is spherically homogeneous, if all source vertices of the edges $e \in \mathcal E \setminus \mathcal F_G$ are leaves of the subtree, and if all of the leaves
 of the subtree have the same out-degree in $G$ (see \cite{Kr5}).
In Section 5 we derive for graphs $G = G(\mathcal V, \mathcal E)$ such that    
$G(\mathcal V, \mathcal F_G)$ is a tree of height two,
a criterion for spherical homogeneity in terms of invariants of topological conjugacy of the Markov-Dyck shifts of the graphs. For  spherically homogeneous graphs 
$G = G(\mathcal V, \mathcal E)$ such that    
$G(\mathcal V, \mathcal F_G)$
has height two we show, that the topological conjugacy class of their Markov Dyck shifts determines the isomorphism class of the graphs.

In Setion 6 
we consider the graphs $G = G(\mathcal V, \mathcal E)$, such that the semigroup
$\mathcal S(M\negthinspace {\scriptstyle D}({G}))$ 
is the graph inverse semigroup of a two-vertex graph, and such that the graph 
$G(\mathcal V,\mathcal F_G)$ decomposes into a one-edge tree and a one-vertex tree. 
Also for these graphs
we show, that the topological conjugacy class of their Markov-Dyck shifts determines the isomorphism class of the graphs, providing at the same time a characterization of these graphs among the Markov-Dyck shifts by invariants of topological conjugacy.

We use the same method of proof is as in \cite{KM}: We choose canonical models for the graphs $G = G(\mathcal V, \mathcal E)$ and obtain sufficient information from certain topological conjugacy invariants of 
$M\negthinspace {\scriptstyle D}({G})$ to reconstruct the canonical model of $G$. In a preliminary Section 2 we introduce notation and recall the relevant invariants of topological conjugacy, that we use.   

\section{Preliminaries}

We introduce notation.
Given a graph $G = G(\mathcal V, \mathcal E)$ we denote by $\mathcal L(G)$ the set of finite paths in $G$, and by $\mathcal L_n(G)$ the set of finite paths in $G$ of length $n \in \Bbb N$. 
 We set
 $$
 \mathcal L_n(V) = \{b \in  \mathcal L_n(G): s(b) = V\}, \quad V \in \mathcal V.
 $$
 We associate to $V \in \mathcal V$
  graphs $G(\mathcal U_V(\eta)  ,  \mathcal B_V(\eta) ),$  
 by
$$
\mathcal U_V(\eta) = \bigcup_{\xi \leq \eta}\{ t(b): b \in   \mathcal L_\eta(G)     \}, \  \
\mathcal B_V(\eta) = \{ e \in \mathcal E: t(e) \in \mathcal U_V(\eta)    \}, \quad \eta \in \Bbb N,
$$
a code
$
\mathcal C_V = \bigcup_{I \in \Bbb N}\mathcal C_V(I)
$
by
\begin{multline*}
\mathcal C_V(I )= ( \bigcap_{1 \leq J  < I}\{  (w_i)_{1 \leq  i \leq 2I} \in \mathcal L(G(\mathcal E^- \cup \mathcal E^	+)):
  \prod_{1 \leq i \leq 2J } w_i \neq \bold 1_V\}) 
  \\
  \bigcap  \ \{  (w_i)_{1 \leq  i \leq 2I} \in \mathcal L(G(\mathcal E^- \cup \mathcal E^	+)): \prod_{1 \leq i \leq 2I } w_i= \bold 1_V\},  \quad I \in \Bbb N,
\end{multline*}
and subsets $\mathcal D_V(\eta)$ of $\mathcal C_V(\eta) $ by
$$
\mathcal D_V(\eta) =
 \{(\prod_{1 \leq k \leq \eta }d^-_k   )
 (\prod_{\eta \geq k \geq 1} d^+_k   ): 
(d_k)_{1 \leq k \leq \eta} \in  \mathcal L_\eta(G)  \},\quad \eta \in \Bbb N.
$$

  Given a graph $G = G(\mathcal V, \mathcal E)$
we denote the edge set of the graph $\widehat{G}$ by $\widehat{\mathcal E}(G)$, and we use $\mathcal R_G$  as its vertex set. Every edge $e \in \mathcal E \setminus \mathcal F_G$ maps into an edge $\widehat{e}$ of $\widehat{G}$, such that 
$
t(\widehat{e}) = t(e)$, 
 and,
 if $s(e) \in \mathcal R_G$,
such  that
$s(\widehat{e}) = s(e)$,  and,
if 
$s(e) \not\in \mathcal R_G$,  
such that $s(e)$ is the root of the subtree, to which $s(e)$  belongs. 
Note, that
$\widehat{\mathcal E} = \{\widehat{e}: e \in  \mathcal E \setminus \mathcal F_G\}.$

For $f \in \mathcal F_G$ we set $\widehat{f}$ equal to $\bold 1_R$, where 
$R\in \mathcal R_G$ is the root of the subtree, to which the edge $f$ belongs. 
We set
$$
\widehat{e^-} = \widehat{e}^- ,\quad     \widehat{e^+} = \widehat{e}^+ ,   \qquad     
e \in \mathcal E \setminus \mathcal F_G.
$$
A periodic point $p$ of period $\pi(p)$ of  $M\negthinspace {\scriptstyle D}({G})$ and its orbit are said to be neutral, if there exists an 
$R\in \mathcal R_G$, which then is uniquely determined by $p$, such that for some 
$i \in \Bbb Z$
$$
\prod_{i \leq j < i + \pi(p)} \widehat p_j  = \bold 1_R.
$$
We denote by $I^{0}_k(M\negthinspace {\scriptstyle D}({G}))$ the cardinality of the set of neutral periodic orbits of length $k$ of $M\negthinspace {\scriptstyle D}({G})$.

For  a non-neutral periodic point $p$ of period $\pi$ of  
$M\negthinspace {\scriptstyle D}({G})$ 
there exists  a simple cycle 
$(\widehat {a}_\ell)_{0 \leq \ell < L}$
in the graph $\widehat {\mathcal E}= G(\mathcal R_G  ,\widehat{\mathcal E} )$, such that for some $i \in \Bbb Z$ either
\begin{align*}
\prod_{i \leq j < i + \pi(p)} \widehat {p_j}  = (\prod_{0 \leq \ell < L}\widehat {a}_\ell^-)^M, 
 \tag {2.1}
\end{align*}
or
\begin{align*}
\prod_{i \leq j < i + \pi(p)} \widehat {p_j}  = (\prod_{L>\ell \geq 0}\widehat {a}_\ell^+)^M.\tag {2.2}
\end{align*}
The simple cycle $\widehat {a}$ is unique up to a cyclic permutation of its edges 
\cite[Section 2]{HI}\cite[Section 4]{HIK}. Following 
the terminology, that was introduced in
\cite{HI}, 
 we refer to the equivalence class of simple cycles in 
 the graph $\widehat {\mathcal E}= G(\mathcal R_G  ,\widehat{\mathcal E} )$, that is assigned in this way to the periodic point $p$, as the multiplier of $p$.
In the case (2.1)((2.2)) the periodic point $p$ is said to have a negative (positive) multiplier.
 The length of the multiplier is $L$. 
 A topological conjugacy carries neutral periodic points into neutral periodic points. Also the map, that assigns to a non-neutral periodic point its multiplier, is an invariant of topological conjugacy \cite[Section 4]{HIK}. 

Given a directed graph $G=G(\mathcal V, \mathcal E)$ we denote the set of multipliers of 
$M\negthinspace {\scriptstyle D}({G})$ by 
$\mathcal M(M\negthinspace {\scriptstyle D}({G}))   $, 
and the set multipliers of 
$M\negthinspace {\scriptstyle D}({G})$ of length $L$ of 
$M\negthinspace {\scriptstyle D}({G})$ by 
$\mathcal M_L(M\negthinspace {\scriptstyle D}({G}) )$.
We denote by $I^{0}_{2k}(M\negthinspace {\scriptstyle D}({G}))$ the cardinality of the set of periodic orbits of length $2k$ of $M\negthinspace {\scriptstyle D}({G})$ with negative multipliers.
We denote 
for a multiplier $\mu \in  \mathcal M(M\negthinspace {\scriptstyle D}({G}))$, by 
$
\mathcal O^{(\mu)}_{k}(M\negthinspace {\scriptstyle D}({G}))
$
the set of orbits with negative multiplier $\mu$  of length $k$, and we set
$$
I^{(\mu)}_k(M\negthinspace {\scriptstyle D}({G})) = \card (\mathcal O^{(\mu)}_{k}(M\negthinspace {\scriptstyle D}({G}))).
$$

In sections 3 - 5 we consider the case of a  graph 
$G=(\mathcal V, \mathcal E)$ 
such that $G(\mathcal V,\mathcal F_G)$ is a tree. In this case
$$
\mathcal M_1(M\negthinspace {\scriptstyle D}({G})) = \widehat{\mathcal E}(G).
$$
We set
$$
\nu (M\negthinspace {\scriptstyle D}({G})) = \card 
(\widehat{\mathcal E}(G)).
$$
Contracting the tree $G(\mathcal V,\mathcal F_G)$  to its root yields the directed graph with a single vertex and $\nu (M\negthinspace {\scriptstyle D}({G}))$ loops.
It follows from \cite[Section 5] {HIK},
that the graphs $G=(\mathcal V, \mathcal E)$, such that the graph 
$G(\mathcal V,\mathcal F_G)$ is a tree, are
precisely the directed graphs that have a Dyck inverse monoid associated to them,
and that 
$$
\mathcal S(M\negthinspace {\scriptstyle D}({G})) = 
\mathcal D_{\nu(M\negthinspace {\scriptscriptstyle D}({G}))}.
$$ 

We denote the root of the tree $G(\mathcal V,\mathcal F_G)$ by $V(0)$.
We denote the level of a vertex  $V \in \mathcal V$ in the tree 
$G(\mathcal V,\mathcal F_G)$ by  
$\lambda(V)$.
A vertex $V \in \mathcal V$  determines a path
$$
f(V) =(f_l(V))_{1 \leq l \leq  \lambda(V) },
$$
from $V(0)$ to $V$.
We set 
\begin{align*}
& f_l(l)  =  f_l(s(e))   , \ \ V_l^{{(e)}} = t(f_l(s(e))),  \quad 1 \leq l \leq \lambda (s(e)), 
\\
&V_0^{{(e)}}= V(0),  \
\mathcal V^{(e)} = \{  V_l^{(e)}: 0 \leq l \leq  \lambda (s(e)) \}, \qquad  \
e \in \mathcal E \setminus \mathcal F_G,.
\end{align*}
We set
\begin{align*}
&\Lambda^{(\widehat{e})}(M\negthinspace {\scriptstyle D}({G})) = \min \{ k \in \Bbb N: 
I^{(\widehat{e})}_k(M\negthinspace {\scriptstyle D}({G})) > 0 \}, \quad \widehat{e} \in \widehat {\mathcal E}(G).
 \end{align*} 
We will use of the notation  \
$\Lambda(M\negthinspace {\scriptstyle D}({G}))$ 
as a means to indicate, that all \
$ \Lambda^{(\widehat{e})}(M\negthinspace {\scriptstyle D}({G}))$,\ $
\widehat{e} \in \widehat {\mathcal E}(G),$
are equal, and that $\Lambda(M\negthinspace {\scriptstyle D}({G}))$  is their common value.
We will also use of the notation 
$I_k(M\negthinspace {\scriptstyle D}({G}))$ 
as a means to indicate, that all 
$ I_{2k}^{(\widehat{e})}(M\negthinspace {\scriptstyle D}({G})),
\widehat{e} \in \widehat {\mathcal E},$
are equal and  that $I_{2k}(M\negthinspace {\scriptstyle D}({G}))$, is their common value, 
$k \in \Bbb N.$

We also set
$$
\Delta^{(\widehat{e})} = I^{(\widehat{e})}_{\Lambda^{(\widehat{e})} + 2} -
 \Lambda^{(\widehat{e})}, \qquad \widehat{e} \in \widehat {\mathcal E}(G).
$$
We note that 
\begin{align*}
I^{0}_2(M\negthinspace {\scriptstyle D}({G})) = \card (\mathcal E) =\nu(M\negthinspace {\scriptstyle D}({G})) + \card (\mathcal F_G). \tag {A}
\end{align*}

The orbit
$$
o^{(\widehat {e}) }= (f^-(s(e)), {e}^-), 
$$
is the only orbit in $ \mathcal O^{(\widehat {e}) } _{\Lambda (\widehat {e})}(M\negthinspace {\scriptstyle D}({G})), \widehat{e} \in \widehat {\mathcal E}(G)$.
By inserting words $w_V \in \mathcal C^{*}_{V}, V \in \mathcal V^{(e)} $, such that
\begin{align*}
\sum_{V \in \mathcal V^{(e)}} \ell(w_V) = 2\eta, \tag{2.3}
\end{align*}
into the orbit $o^{(\widehat {e}) }$ one obtains an orbit
$$
( (w_{V_{l - 1}} f_l^-(s(e)))_{1 \leq  l \leq \lambda(e)}, w_{V_{s(e)}} , e )    \in \mathcal O^{(\widehat {e}) } _{\Lambda (\widehat {e})+ 2\eta}, \qquad \eta \in \Bbb N.
$$
Conversely, every orbit 
$
o \in  
\mathcal O^{(\widehat {e}) } _{\Lambda (\widehat {e})+ 2\eta}, \eta \in \Bbb N,
$
can be obtained from a unique set of words 
$w_V \in \mathcal C^{*}_{V}, V \in \mathcal V^{(e)} $, 
such that (2.3) holds:
 The word $w_V$ is identified as the longest word  $w$ in 
$\mathcal C_V^{*}$, such that $o$ contains a word in 
$(\mathcal E^-\setminus \mathcal F_G) w 
\mathcal (\mathcal E^-\setminus \mathcal F_G)$. 


\section{The Family $\bold{F}_{IV}$.}

Set
$$
\Pi = \{ (H, h, h_0, h_1) \in  \Bbb N: h < H,  h_0 \leq h_1 \leq H - h\}.
$$
Given data
$$
 (H, h, h_0, h_1) \in \Pi
$$
 we build a directed   $G(H, h, h_0, h_1)$. We set
 \begin{multline*}
\mathcal V (H, h, h_0, h_1) = 
 \{V(i): 0 \leq i \leq h\} \  \cup
\left
( \bigcup_{\alpha \in \{ 0,1\}}\{V_\alpha(i): 0 \leq i \leq h_\alpha\}  
\right
) \cup 
\\
\left
( \bigcup_{\alpha,\beta \in \{ 0,1\}}\{V_{\alpha,\beta}(i): 
0 \leq i 
\leq H - h - h_\alpha\}  
\right),
\end{multline*}
\begin{multline*}
\mathcal F (H, h, h_0, h_1) = \{f(i): 0 \leq i \leq h   \}  \cup 
\left
( \bigcup_{\alpha \in \{ 0,1\}}\{f_\alpha(i): 0 \leq i \leq h_\alpha\}  
\right
) \cup
 \\
\left
( \bigcup_{\alpha,\beta \in \{ 0,1\}}\{f_{\alpha,\beta}(i):
 0 \leq i 
\leq H - h - h_\alpha\}  
\right),
\end{multline*}
\begin{align*}
 \mathcal E  = 
 \{ e_{\alpha,\beta}:\alpha,\beta \in \{ 0,1\} \} ,
 \end{align*}
and we define the graph $G(H, h, h_0, h_1)$ as the directed graph with vertex set 
$\mathcal V (H, h, h_0, h_1)$
and edge set  $\mathcal F (H, h, h_0, h_1) \cup\mathcal E  = 
 \{ e_{\alpha,\beta}:\alpha,\beta \in \{ 0,1\} \}$, 
and source and target mappings given by
\medskip
$$
s(f(i)) = V(i - 1), t(f(i)) = V(i),   \quad 0 < i \leq h,
$$
\medskip
\begin{align*}
&s(f_\alpha(1)) = V(h),
\\
&s(f_\alpha(i)) = V_\alpha(i - 1),  \quad 1 < i \leq h_\alpha,
\\
&t(f_\alpha(i)) = V_\alpha(i), 
  \quad 0 < i \leq h_\alpha, \quad \alpha \in \{ 0,1\},
\end{align*}
\medskip
\begin{align*}
&s(f_{\alpha,\beta}(1)) = V_\alpha(h),
\\
&s(f_{\alpha,\beta}(i)) = V_\alpha(i - 1),  \quad 
1 < i \leq H- h - h_\alpha,
\\
&t(f_{\alpha,\beta}(i)) = V_{\alpha,\beta}(i), 
  \quad 0 < i \leq H - h - h_\alpha, \quad \alpha, \beta \in \{ 0,1\},
\end{align*}
\medskip
\begin{align*}
&s(e_{\alpha,\beta}) =
\begin{cases}
V_{\alpha}(h_\alpha), 
&\text{if $H - h - h_\alpha = 0$}, \\
V_{\alpha,\beta}(h_{\alpha,\beta}), &\text{if  $H - h - h_\alpha > 0$},
\end{cases} 
\\
&t(e_{\alpha,\beta}) = V(0), \qquad \quad  \ \alpha, \beta \in \{ 0,1\}.
\end{align*}
The graph $G(H, h, h_0, h_1)$ has the single subtree 
$G(\mathcal V (H, h, h_0, h_1), \mathcal F (H, h, h_0, h_1) )$.

We set
\begin{multline*}
\Pi_{IV} =
\{ (H ,h,  h_0, h_1  ) \in \Bbb N^4:
\\
h < H,  h_0 \leq h_1\leq H-h-  h_0, 3h +h_0 +h_1\geq 
2H +  \left \lceil \tfrac{H}{2} \right \rceil 
\end{multline*}
and define the family $\bold F_{IV}$ by
$$
\bold F_{IV} = \{ G(H , h, h_0, h_1  )  :(H ,h, h_0, h_1  ) \in \Pi_{IV}  \}.
$$

For $(H ,h, h_0, h_1  ) \in \Pi_{IV} $ one has 
\begin{align*}
h > \frac{H}{2}.  \tag {B}
\end{align*}

\begin{theorem}
For a finite directed graph $G=G(\mathcal V, \mathcal E)$ there exist data
$$
(h,h_0,h_1,H)   
\in \Pi_{IV},
$$
such that
there is a topological conjugacy 
\begin{align*}
M\negthinspace {\scriptstyle D}({{G}}) \simeq 
M\negthinspace {\scriptstyle D}(G(H,h,h_0,h_1)), 
\end{align*}
if and only if  
\begin{align*}
\mathcal S(M\negthinspace {\scriptstyle D}({{G}})) = \mathcal D_4,\tag {C1}
\end{align*}
\begin{align*}
\Lambda ^{(e)}  = H + 1, \
\Delta^{(e)} = 2,    \qquad  e \in
\mathcal M(M\negthinspace {\scriptstyle D}({{G}})),
\tag {C2}
\end{align*}
and
\begin{align*}
\tau(M\negthinspace {\scriptstyle D}({{G}}))  \leq \Lambda (M\negthinspace {\scriptstyle D}({{G}}))  - 1+
 \left\lceil \tfrac{\Lambda(M\negthinspace {\scriptstyle D}({{G}}))  - 1}{2}\right\rceil  .
\tag {C3}
\end{align*}
\end{theorem} 

\begin{proof} 
The invariant conditions (C 1 - 2) are the translation of a description of the graphs 
$G(H,h,h_0,h_1), (H,h,h_0,h_1) \in \Pi$, that is in terms of the subtree, of its height, its number of its leaves and its branch points. The invariant condition (C 3) is the translation of the defining condition of 
$\Pi_{IV}$.
\end{proof}

\begin{theorem}
For directed graphs $G=G(\mathcal V  , \mathcal E )$, such that 
$\mathcal S(M\negthinspace {\scriptstyle D}({{G}}))$
 is  a Dyck inverse monoid,
 and such that (D 1 - 2) hold,
the topological  conjugacy of the Markov-Dyck shifts 
$M\negthinspace {\scriptstyle D}({{G}})$ implies the isomorphism of the graphs 
$G$.
\end{theorem}
\begin{proof}
Let 
$$  
(H,h,h_0,h_1) \in \Pi_{IV}.
$$
We introduce notation for subsets of the vertex set of a graph 
$$
G = G(H,h,h_0,h_1)
$$
 and of the graph 
 $$ \widetilde G = G(H, \left\lceil \tfrac{H}{2}\right\rceil  , \left \lfloor \tfrac{H}{2} \right \rfloor ,  \left \lfloor \tfrac{H}{2} \right \rfloor ),
$$
and for mappings between these subsets. (We put a tilde into the notations that pertain to 
the graph $\widetilde G$.)
 
We set
$$
\mathcal U^{\blacktriangle}(G) = \{V(h - i): 0 \leq i < h_0 \}, \quad
\widetilde{\mathcal U}^\blacktriangle(\widetilde {G}) = 
\{ \widetilde {
V} (\left\lceil \tfrac{H}{2}\right\rceil  - i ) : 0 \leq i < h_0       \},
$$
and we define a bijection
$$
\Phi_G:  \mathcal U^\blacktriangle(G)  \to \widetilde{\mathcal U}^\blacktriangle(\widetilde {G})
$$
by
$$ 
\Phi_G(V(h - i)) = 
 \widetilde {V} 
(\left\lceil \tfrac{H}{2}\right\rceil   - i ),
 \quad 
0 \leq i < h_0.
$$

We set
$$
\mathcal U^{\blacktriangledown}_\alpha = \{ V_\alpha(i): 0 < i \leq h_0\},  \ \
\widetilde {\mathcal U}^{\blacktriangledown}_\alpha = \{ \widetilde {V}_\alpha(i): 0 < i \leq h_0\}, \qquad \alpha \in \{ 0, 1 \}.
$$
One has, that
$$
\mathcal U_{V(h - h_0 +1   )}(2h_0 - 1) = \mathcal U(G) \cup 
\mathcal U^{\blacktriangledown}_{0}(G) \cup\mathcal U^{\blacktriangledown}_{1}(G),
$$
$$
\mathcal U_{\widetilde {V}(h - h_0 +1   )}(2h_0 - 1) = \widetilde {\mathcal U}(G) \cup 
\widetilde {\mathcal U}^{\blacktriangledown}_{0}(G) \cup\widetilde {\mathcal U}^{\blacktriangledown}_{1}(G). 
$$
We extend the bijection
$$
\Phi_G:  \mathcal U^\blacktriangle(G)  \to \widetilde{\mathcal U}^\blacktriangle(\widetilde {G})
$$
to a bijection
$$
\widehat{\Phi}_G: \mathcal U_{V(h - h_0 +1   )}(2h_0 - 1) \to 
\mathcal U_{\widetilde {V}(h - h_0 +1   )}(2h_0 - 1)
$$
by
$$
\widehat{\Phi}_G(V_\alpha(i)) = \widetilde{V}_\alpha(i), \qquad 0 < i \leq h_0, \ \alpha \in \{0,1  \}.
$$

We set
$$
\mathcal U^{\blacktriangle}_\alpha(G) = \{V_\alpha(h_\alpha - i): 0 \leq i < h_0 \}, \quad
\widetilde{\mathcal U}^{\blacktriangle}_\alpha(\widetilde {G}) = 
\{ \widetilde {
V}_\alpha ( \left \lfloor \tfrac{H}{2} \right \rfloor  - i ) : 0  \leq h_0   < h_0   \},
$$
and we define a bijection
$$
\Phi^{(\alpha)}_G:  \mathcal U^{\blacktriangle}_\alpha(G)  \to \widetilde{\mathcal U}^{\blacktriangle}_\alpha(\widetilde {G})
$$
by
$$ 
\Phi^{(\alpha)}_G(V_\alpha(h_\alpha - i)) = 
 \widetilde {V}_\alpha 
( \left \lfloor \tfrac{H}{2} \right \rfloor  - i ), \
0 \leq i < h_0, \qquad \alpha\in \{0,1 \}.
$$

We set
$$
\mathcal U^{\blacktriangledown}_{\alpha}(G) = \{V_\alpha(i): 0 < i \leq h_0 \},
 \qquad \alpha\in \{0,1 \}.
$$
(Note that $ \mathcal U^{\blacktriangledown}_{0}(G) = \mathcal U^{\blacktriangle}_{0}(G)$. If $h_1 = h_0$, then also $ \mathcal U^{\blacktriangledown}_{1}(G) = \mathcal U^{\blacktriangle}_{1}(G)$.) 
We set
$$
\widetilde{\mathcal U}^{\blacktriangledown}(\widetilde{G})  
= \{\widetilde{V}(i): 0 \leq i <h_0\}.
$$
One has, that
$$
{\mathcal U}
_{\widetilde{V}_\alpha( {\left \lfloor \tfrac{\scriptscriptstyle H}{\scriptscriptstyle 2} \right \rfloor } - h_0 +1   )}(2h_0 - 1) =
\widetilde{\mathcal U}^{\blacktriangle}_\alpha(\widetilde{G}) \cup\widetilde{\mathcal U}^{\blacktriangledown}(\widetilde{G}), \qquad \alpha \in \{0,1  \}.
$$

By (B) it is possible to set
$$
\mathcal U^{\blacktriangledown} (G) = \{V(i): 0 \leq i < h_0\}.
$$

To determine  sets $\mathcal U_{\alpha, \beta},  \beta \in \{0, 1 \}$, such that
$$
\mathcal U_{V_\alpha(h_\alpha - h_0 +1)}(2h_0 - 1) =
 \mathcal U^{\blacktriangle}_\alpha(G)
 \cup  \mathcal U_{\alpha, 0 } (G) \cup\mathcal U_{\alpha, 1 } (G), 
$$
and to extend  
the bijections 
$$
\Phi^{(\alpha)}_G:  \mathcal U^{\blacktriangle}_\alpha(G)  \to \widetilde{\mathcal U}^{\blacktriangle}_\alpha(\widetilde {G})
$$
to  surjections
$$
\widehat{\Phi}_G^{(\alpha)}:\mathcal U_{V_\alpha(1)}(2h_0 - 1)  \to {\mathcal U}
_{\widetilde{V}_\alpha( {\left \lfloor \tfrac{\scriptscriptstyle H}{\scriptscriptstyle 2} \right \rfloor } - h_0 +1   )}(2h_0 - 1), \quad\alpha \in \{0, 1 \},
$$
we distinguish cases:

In the case, that $h_\alpha = H - h$, we set
\begin{align*}
&\mathcal U_{\alpha, \beta} = \mathcal U^{\blacktriangledown} (G), \quad\beta \in \{0, 1 \},
\\
&\widehat{\Phi}_G^{(\alpha)} (V(i)) = \widetilde{V}(i) , \qquad 0 \leq i < h_0.
\end{align*}

In the case, that $H - h - h_0 < h_\alpha < H - h$, we set
\begin{multline*}
\mathcal U_{\alpha, \beta} = \{V_{\alpha, \beta}(i): 0 < i \leq H - h - h_\alpha   \} \cup
\{V(i): 0 \leq i < h_0 + h +h_\alpha -H   \}, 
\\
\beta \in \{0, 1 \},
\end{multline*}
\begin{align*}
&\widehat{\Phi}_G^{(\alpha)}( V_{\alpha, \beta}(i)  ) =  \widetilde{V}(i), \qquad 0 < i \leq H - h - h_\alpha,
\\
&\widehat{\Phi}_G^{(\alpha)}( V(i)  ) =  \widetilde{V}(i + H - h - h_\alpha), \qquad 
0\leq i < h_0 + h + h_\alpha - H.
\end{align*}

In the case, that $h_\alpha \leq H - h - h_0$, we set
\begin{align*}
&\mathcal U_{\alpha, \beta} = \{V_{\alpha, \beta}(i): 0 < i \leq  h_0   \}, \quad \beta \in \{0, 1 \},
\\
&\widehat{\Phi}_G^{(\alpha)}(V_{\alpha, \beta}(i) ) = \widetilde{V}(i - 1), \qquad 0 < i \leq h_0.
\end{align*}

Supplementing the bijection
 $$
 \widehat {\Phi}_G: \mathcal U_{V(h - h_0 +1   )}(2h_0 - 1) \to
{\mathcal U}
_{\widetilde{V}( {\left \lfloor \tfrac{\scriptscriptstyle H}{\scriptscriptstyle 2} \right \rfloor } - h_0 +1   )}(2h_0 - 1) 
$$
with the bijection 
$$
\Psi_G: \mathcal B_{V(h - h_0 +1   )}(2h_0 - 1) \to
{\mathcal B}
_{\widetilde{V}( {\left \lfloor \tfrac{\scriptscriptstyle H}{\scriptscriptstyle 2} \right \rfloor } - h_0 +1   )}(2h_0 - 1), 
$$  
that is given by
$$
t(\Psi_G{(b)}) = \Phi_G(t(b)). \qquad  b \in \mathcal B_{V(h - h_0 +1   )}(2h_0 - 1),
$$
yields an isomorphism between the graph 
$$
G( \mathcal U_{V(h - h_0 +1   )}(2h_0 - 1)  ,   \mathcal B_{V(h - h_0 +1   )}(2h_0 - 1)    )
$$
and the graph
$$
G({\mathcal U}
_{\widetilde{V}( {\left \lfloor \tfrac{\scriptscriptstyle H}{\scriptscriptstyle 2} \right \rfloor } - h_0 +1   )}(2h_0 - 1)  ,   {\mathcal B}
_{\widetilde{V}( {\left \lfloor \tfrac{\scriptscriptstyle H}{\scriptscriptstyle 2} \right \rfloor } - h_0 +1   )}(2h_0 - 1)    ).
$$
Restricting this isomorphism yields isomorphisms between the graph
$$
G(   \mathcal U_{V(h - \eta + i  )} (\eta) ,  \mathcal B_{V(h - \eta + i  )} (\eta)    )
$$
and the graph
$$
 {G}(     {\mathcal U}_{ \widetilde{V}(\left \lceil \tfrac{\scriptscriptstyle H}{\scriptscriptstyle 2} \right \rceil  - \eta + i  )} (\eta)      ,  {\mathcal B}_{ \widetilde{ V}(\left \lceil \tfrac{\scriptscriptstyle H}{\scriptscriptstyle 2} \right \rceil  - \eta + i  )} (\eta)   ) ,   \qquad  0 < i \leq  \eta \leq h_0,
$$
 and these isomorphisms induce isomorphisms between the graph
$$
G(   \mathcal U_{V(h - \eta + i  )} (\eta) ,  \mathcal B^-_{V(h - \eta + i  )} (\eta)\cup
 \mathcal B^+_{V(h - \eta + i  )} (\eta   ))
$$
and the graph
$$
{G}(     {\mathcal U}_{ \widetilde{V}(\left \lfloor \tfrac{\scriptscriptstyle H}{\scriptscriptstyle 2} \right \rfloor  - \eta + i  )} (\eta)      ,  {\mathcal B}^-_{ \widetilde{V}(\left \lfloor \tfrac{\scriptscriptstyle H}{\scriptscriptstyle 2} \right \rfloor  - \eta + i  )} (\eta) \cup
 {\mathcal B}^+_{ \widetilde{V}(\left \lfloor \tfrac{\scriptscriptstyle H}{\scriptscriptstyle 2} \right \rfloor  - \eta + i  )} (\eta  )) ,   \qquad  0 < i \leq  \eta \leq h_0,
$$
that produce a length preserving bijection between the set of words in
$\mathcal C_{V(h - \eta + i)}$ of length  $2\eta$ and 
the set of words in
$\mathcal C_{\widetilde{V}(\left \lfloor \tfrac{\scriptscriptstyle H}{\scriptscriptstyle 2} \right \rfloor   - \eta + i)}$ of length $2\eta$, $0 < i \leq  \eta \leq h_0$.
In particular one has, that
\begin{align*}
 \card(\mathcal C_V(\xi))
= 
 \card(\widetilde{\mathcal C}_{\Phi_G(V)}(\xi)), \quad
 V \in \mathcal U(G), \    1 \ \leq \xi \leq \eta \leq h_0. \tag {3.1}
\end{align*}

Supplementing the surjection
$$
\widehat {\Phi}_G^{(\alpha)}: \mathcal U^{(\alpha)}_{V(h - h_0 +1   )}(2h_0 - 1) \to
{\mathcal U}
_{\widetilde{V}( {\left \lfloor \tfrac{\scriptscriptstyle H}{\scriptscriptstyle 2} \right \rfloor } - h_0 +1   )}(2h_0 - 1), \qquad \alpha \in \{0, 1\}, 
$$
with the surjection
$$
\Psi_G^{(\alpha)}: \mathcal B_{V(h - h_0 +1   )}(2h_0 - 1) \to
{\mathcal B}
_{\widetilde{V}( {\left \lfloor \tfrac{\scriptscriptstyle H}{\scriptscriptstyle 2} \right \rfloor } - h_0 +1   )}(2h_0 - 1), 
$$  
that is given by
\begin{align*}
\Psi_G^{(\alpha)} ( f_{\alpha,\beta}  ) = \widetilde{e}_{\alpha,\beta} , \qquad \beta \in \{0, 1\}, 
 \tag{3.2}
\end{align*}
and by
\begin{align*}
&\Psi_G^{(\alpha)}(f_{\alpha,\beta}(i))  
= \widetilde{f}(i), \quad 1 \leq i \leq H - h - h_0, 
\\
&\Psi_G^{(\alpha)}(e_{\alpha,\beta}) =
  \widetilde{f}(H - h -h_0 +1), \qquad\beta \in \{0, 1\},
\end{align*}
and
\begin{align*}
\Psi_G^{(\alpha)}(f(i)) = \widetilde{f}(i + 2h_0 - h - H), \quad 0 < i <H - h - h_0,
\end{align*}
yields an isomorphism between the graph 
$$
G( \mathcal U_{V(h - h_0 +1   )}(2h_0 - 1)  ,   \mathcal B_{V(h - h_0 +1   )}(2h_0 - 1)    )
$$
and the graph
$$
G( {\mathcal U}
_{\widetilde{V}( {\left \lfloor \tfrac{\scriptscriptstyle H}{\scriptscriptstyle 2} \right \rfloor } - h_0 +1   )}(2h_0 - 1)  ,   {\mathcal B}
_{\widetilde{V}( {\left \lfloor \tfrac{\scriptscriptstyle H}{\scriptscriptstyle 2} \right \rfloor } - h_0 +1   )}(2h_0 - 1)    ).
$$
Restricting this isomorphism yields homomorphisms between the graph
$$
G(   \mathcal U_{V_0(h_0 - i  )} (\eta) ,  \mathcal B_{V_0(h_0 - i  ) )} (\eta)    )
$$
and the graph
$$
 {G}(     \widetilde{\mathcal U}_{V_0(\left \lfloor \tfrac{\scriptscriptstyle H}{\scriptscriptstyle 2} \right \rfloor  - \eta + i  )} (\eta)      ,   \widetilde{\mathcal B}_{V_0(\left \lfloor \tfrac{\scriptscriptstyle H}{\scriptscriptstyle 2} \right \rfloor  - \eta + i  )} (\eta)   ) ,   \qquad  0 < i \leq  \eta \leq h_0,
$$
and these isomorphisms induce isomorphisms between the graph
$$
G(   \mathcal U_{V_0(h_0 -  i  )} (\eta) ,  \mathcal B^-_{V(h_0 -  i  )} (\eta)\cup
 \mathcal B^+_{(h_0 -  i  )} (\eta   ))
$$
and the graph
$$
 {G}({\mathcal U}_{ \widetilde{V}(\left \lfloor \tfrac{\scriptscriptstyle H}{\scriptscriptstyle 2} \right \rfloor  - \eta + i  )} (\eta)      ,  
 {\mathcal B}^-_{ \widetilde{V}
 (\left \lfloor \tfrac{\scriptscriptstyle H}{\scriptscriptstyle 2} \right \rfloor  - \eta + i  )} (\eta) \cup
 {\mathcal B}^+_{ \widetilde{V}(\left \lfloor \tfrac{\scriptscriptstyle H}{\scriptscriptstyle 2} \right \rfloor  - \eta + i  )} (\eta  )) ,   \qquad  0 < i \leq  \eta \leq h_0,
$$
which yield,  as a consequence of (3.2),
a length preserving bijection between the set of words in
$\mathcal C_{V(h - \eta + i)}$ of length less than or equal to $2\eta$ and 
the set of words in
$\mathcal C_{\widetilde{V}(\left \lfloor \tfrac{\scriptscriptstyle H}{\scriptscriptstyle 2} \right \rfloor   - \eta + i)}$ of length less than or equal to $2\eta$, $0 < i \leq  \eta \leq h_0$.
In particular one has, that
\begin{align*}
 \card( \mathcal C_V (\xi)            \})
= 
 \card(\{\widetilde {\mathcal C}_{\Phi_G^{(\alpha)}(V)}(\xi))   ,
 \tag {3.3}
 \ \ V \in \mathcal U^{{(\alpha)}}(G), \   1 \leq \xi \leq \eta \leq h_0.
\end{align*}

Let $ \alpha,\beta ,\widetilde{\alpha},\widetilde{\beta} \in \{0, 1  \}$. Choose a bijection
$\bar \Phi : \mathcal V^{(e_{\alpha,\beta})} \to 
 \mathcal V^{(\widetilde{e}_{\widetilde{\alpha},\widetilde{\beta}})}$ such that
$$
\bar \Phi \restriction  \mathcal U(G) = \Phi_G,  \ \
\bar \Phi \restriction \mathcal U^{(\alpha)}(G) = \Phi_G^{(\alpha)}.
$$
A vertex $V \in \mathcal V^{(e_{\alpha,\beta})}$ is in 
$
\mathcal V^{(e_{\alpha,\beta})}  \setminus( \mathcal U(G)  \cup  \mathcal U^{(\alpha)}(G)  )
$
 if and only if the graph $G(\mathcal U_V (h_0) ,\mathcal D_V(h_0)  )$ is a tree with one leaf,
 and 
 a vertex $\widetilde{V} \in \mathcal V^{(\widetilde{e}_{\widetilde{\alpha},\widetilde{\beta}})}$ is in 
$
\mathcal V^{(\widetilde{e}_{\widetilde{\alpha},\widetilde{\beta}})}  \setminus(\widetilde{ \mathcal U}(\widetilde{G})  \cup  {\mathcal U}^{(\widetilde{\alpha})}(\widetilde{G})  )
$
 if and only if the graph $G(\mathcal U_{\widetilde{V}} (h_0) ,\mathcal D_{\widetilde{V}}(h_0)  )$ is a tree with one leaf. 
 As a consequence it follows from (3.1) and (3.3) that
\begin{align*}
\card( \mathcal C_V(\xi)) 
= 
 \card(\widetilde{\mathcal C}_{\bar\Phi(V)}(\xi) ) ,
 \ \ V \in \mathcal V^{(e_{\alpha, \beta})},    \quad 1 \leq \xi \leq \eta \leq h_0.
\end{align*}
This implies, that
\begin{align*}
 I^{(e_{\alpha, \beta})}_{H + 1 + 2\eta } =
 I^{{(\widetilde{e}_{\widetilde{\alpha}, \widetilde{\beta}}})}_{H + 1 + 2\eta}
 , \qquad 1 \leq \eta \leq h_0.
 \end{align*}
From this and from the structure of the orbits in $\mathcal O^{(e_{\alpha, \beta})}_{H + 1 + 2(h_0 + 1)}$
( $\mathcal O^{(e
_{\widetilde{e}_{\widetilde{\alpha}, \widetilde{\beta}}}))}_{H + 1 + 2(h_0 + 1)} $) it is seen, that
\begin{align*}
I^{(e_{\alpha, \beta})}_{H + 1 + 2(h_0 + 1)}
 = I^{(e_{\alpha, \beta})}_{H + 1 + 2h_0 }+ 
 \sum_{V \in \mathcal V^{{(e_{\alpha, \beta})}}}
  \card(\mathcal D_V^{(h_0 + 1)})  = &
  \\
 I^{(\widetilde{e}_{\widetilde{\alpha}, \widetilde{\beta}})}_{H + 1 + 2h_0 }+
  \sum_{V \in \mathcal V^{{(e_{\alpha, \beta})}}}
  \card(\mathcal D_V^{(h_0 + 1)}) = &
  \\
 I^{(\widetilde{e}_{\widetilde{\alpha}, \widetilde{\beta}})}_{H + 1 + 2(h_0 +1) }+
  \sum_{V \in \mathcal V^{{(e_{\alpha, \beta})}}}
  \card(\mathcal D_V^{(h_0 + 1)}&) - \sum_{\widetilde{V} \in \mathcal V^{{(\widetilde{e}_{\widetilde{\alpha}, \widetilde{\beta}})}}}
  \card(\mathcal D_{\widetilde{V}}^{(h_0 + 1)}).
\end{align*}
A path count yields the following result: If $h_1 = h_0$, then 
$$
 \sum_{V \in \mathcal V^{{(e_{\alpha, \beta})}}}
  \card(\mathcal D_V^{(h_0 + 1)}) - \sum_{\widetilde{V} \in \mathcal V^{{(\widetilde{e}_{\widetilde{\alpha}, \widetilde{\beta}})}}}
  \card(\mathcal D_{\widetilde{V}}^{(h_0 + 1)}) = 4,,
$$
and if $h_1 > h_0$, then
$$
 \sum_{V \in \mathcal V^{{(e_{\alpha, \beta})}}}
  \card(\mathcal D_V^{(h_0 + 1)}) - \sum_{\widetilde{V} \in \mathcal V^{{(\widetilde{e}_{\widetilde{\alpha}, \widetilde{\beta}})}}}
  \card(\mathcal D_{\widetilde{V}}^{(h_0 + 1)}) = 2.
$$
It follows, that the parameter $h_0$ can be determined from invariants of topological conjugacy of 
$M\negthinspace {\scriptstyle D}({{G}})$ as follows: 
If 
$$
I^{(e_{\alpha, \beta})}_{H + 1 + 2\eta} =
I^{(\widetilde{e}_{\widetilde{\alpha}, \widetilde{\beta}})}_{H + 1 + 2\eta }
, \quad 1\leq \eta \leq   \left \lfloor \tfrac{H}{2} \right \rfloor, 
$$
then $G = \widetilde{G}$. Otherwise
\begin{align*}
h_0 = \min \{  \eta \in [1,  \left \lfloor \tfrac{H}{2} \right \rfloor) : 
I^{(e_{\alpha, \beta})}_{H + 1 + 2(\eta +1)} >
I^{(\widetilde{e}_{\widetilde{\alpha}, \widetilde{\beta}})}_{H + 1 + 2(\eta + 1) } \}. \tag {3.4}
\end{align*}
If 
$$
I^{(e_{\alpha, \beta})}_{H + 1 + 2(\eta +1)} =
I^{(\widetilde{e}_{\widetilde{\alpha}, \widetilde{\beta}})}_{H + 1 + 2(\eta + 1) } + 4, 
$$
then $h_1 = h_0$. If 
$$
 I^{(e_{\alpha, \beta})}_{H + 1 + 2(\eta +1)} =
I^{(\widetilde{e}_{\widetilde{\alpha}, \widetilde{\beta}})}_{H + 1 + 2(\eta + 1) }  + 2, 
$$
then $h_1 > h_0$.

In the case, that $h_1 > h_0$,  we apply the same argument to determine the value of the parameter $h_1$. The maximal  value of $h_1$, that is compatible with the value for 
$h_0(M\negthinspace {\scriptstyle D}({{G}})))$ as  given by (3.4), is given by
$$
\widetilde{\widetilde {h}}_1 =  \tfrac{5}{2}(\Lambda (M\negthinspace {\scriptstyle D}({{G}})) - 1) - \tfrac{1}{2}\tau (M\negthinspace {\scriptstyle D}({{G}})) - 
2h_0(M\negthinspace {\scriptstyle D}({{G}})).
$$
Set
$$
\widetilde{\widetilde {G}} = G( H, \left \lceil \tfrac{\scriptscriptstyle H}{\scriptscriptstyle 2} \right \rceil, h_0(M\negthinspace {\scriptstyle D}({{G}})), \widetilde{\widetilde {h}}_1 .) 
$$
(We put a double tilde into notations pertaining to $\widetilde{\widetilde {G}}$.) Let $ \alpha,\beta ,\widetilde{\alpha},\widetilde{\beta} \in \{0, 1  \}$.
If
$$
I^{(e_{\alpha, \beta})}_{H + 1 + 2(\eta +1)}  =
 I^{(\widetilde{\widetilde{e}}_{\widetilde{\alpha}, \widetilde{\beta}})}_{H + 1 + 2(\eta + 1) }   , \qquad
    h_0(M\negthinspace {\scriptstyle D}({{G}}))< \eta < \widetilde{\widetilde {h}}_1,
$$
then 
$$
h_1 = \widetilde{\widetilde {h}}_1,
$$
which means, that $G =  \widetilde{\widetilde {G}}  $.
Otherwise  $h_1$ is given by
$$
h_1 = \min \{\eta\in ( h_0, \widetilde{\widetilde {h}}_1):  
I^{(e_{\alpha, \beta})}_{H + 1 + 2(\eta +1)} >
  I^{(\widetilde{\widetilde{e}}_{\widetilde{\alpha}, \widetilde{\beta}})}_{H + 1 + 2(\eta + 1) }\}.
$$

By (A) the parameter $h$ can determined  from invariants of topological conjugacy of 
$M\negthinspace {\scriptstyle D}({{G}})$ by
$$
h = \frac{1}{3}(4\Lambda(M\negthinspace {\scriptstyle D}({{G}})) + \nu(M\negthinspace {\scriptstyle D}({{G}}))- h_0(M\negthinspace {\scriptstyle D}({{G}})) -h_1(M\negthinspace {\scriptstyle D}({{G}})) -
 I_2^{0}(M\negthinspace {\scriptstyle D}({{G}})) - 4). \qed
$$
\renewcommand{\qedsymbol}{}
\end{proof}


\section{The Family $\bold{F}_{V}$}   

Let  $G=G(\mathcal V, \mathcal E)$, be a directed graph, such that $G(\mathcal V, \mathcal F_G)$ is a tree.
We set
 $$
 \beta(V) = D(V (0))+\sum_{1 \leq l \leq  \lambda(V)}D(t(f_l(V))), \qquad V \in \mathcal V.
 $$

We define the family $\bold{F}_{V}$  of directed graphs,  as the family, that contains the 
graphs  $G=G(\mathcal V, \mathcal E)$, such that $G(\mathcal V, \mathcal F_G),$ is a tree 
and that have an edge 
 $e\in  \mathcal E \setminus \mathcal F_G$, such that
\begin{align*}
\beta(s(e))  =\lambda (s (e))+ \card ( \mathcal E \setminus \mathcal F_G ) . \tag {4.1}
\end{align*}
We describe the canonical models that we use for the graphs in this family. 
For $k \in \Bbb N$ and $\ell \in \Bbb Z_+$, such that $\ell > k - 1$ we denote by 
$\Omega_{K, \ell }$ the set of K-tuples
$$
(\eta_k, (\mu_k(L))_{L \in \Bbb Z_+})_{1 \leq k \leq K}\in 
(\Bbb Z_+ \times \Bbb Z_+^{ \Bbb Z_+})^K,
$$
such that
\begin{align*}
&\eta_K \leq \ell,
\\
&\mu_K(L) =0, \qquad L > \ell - \eta _K,
\\
&\eta_k < \eta_{k+1}, \qquad 1 \leq k < K,
\\
&0 < \sum_{L \in \Bbb Z_+}\mu _k(L) ,\qquad 1 \leq k \leq K,
\\
&\sum_{1 \leq k \leq K,L \in \Bbb Z_+}\mu _k(L) < \infty.
\end{align*}
We define a set $\Pi_{IV}$ of data as the set of triples
$
(K, \ell, \omega),
$
with $K\in \Bbb N,\ell \in \Bbb Z_+$, 
such that $ \ell \geq K - 1$, and 
$
\omega 
\in \Omega_{K, \ell }.
$
Given data 
$$
(K, \ell,
\omega)\in \Pi_{IV}, \quad
\omega = (\eta_k, (\mu_k(L))_{L \in \Bbb Z_+})_{1 \leq k \leq K},
$$
we build a directed graph
$
G(K, \ell,  \omega).
$
We set
\begin{align*}
\mathcal V&(K, \ell, \omega) =  
\\
&\{V(h): 0 \leq h \leq   \ell      \} \  \cup 
\bigcup_{1 \leq k \leq K}\{  V_{k_,L}(m,l):1 \leq l \leq L , 1 \leq m \leq\mu_{k}(L),
 L\in \Bbb N   \},
\\
\mathcal F&(K, \ell,\omega) = 
\\
&\{f(h): 1 \leq h \leq    \ell\} \ \cup \bigcup_{1 \leq k \leq K}\{f_{k_,L}(m,l):1 \leq l \leq L , 1 \leq m \leq\mu_{k}(L), L \in \Bbb N \}, 
\\
\\
\mathcal E&^{(k)}(K, \ell, \omega) = 
\{e_{k, L}(m): 1 \leq m \leq \mu_k(L), L \in \Bbb Z_+\}, \quad 1\leq k \leq K,
\\
\mathcal E&(K, \ell, \omega) = 
\bigcup_{1 \leq k \leq K}\mathcal E^{(k)}(K, \ell, \omega)
 \cup \{  {e^\top} \},
\end{align*}
and we define the graph
$
G(K, \ell,  \omega)
$
as the directed graph with vertex set 
$\mathcal V(K, \ell,  \omega)$
 and edge set $\mathcal F(K, \ell,  \omega) \cup 
 \mathcal E(K, \ell,  \omega)$, and with 
source and target mappings, that are given by 
$$
s( f(h)) = V(h-1)  , \quad 1 \leq h\leq \ell , 
$$
$$
t(  f(h)) =   V(h)  ,\qquad \negthinspace \quad 0 \leq h< \ell  ,
$$
$$
s( f_{k_,L}(m,1)    ) = V(\eta_k), \quad 1 \leq m \leq M_k, 1\leq k \leq K,
$$
$$
s(  f_{k_,L}(m,l)  ) =  V_{k_,L}(m,l-1),\quad 1 < l \leq L,1 \leq m \leq \mu_k(L), 1\leq k \leq K,
$$
$$
t(  f_{k_,L}(m,l)  ) =  V_{k_,L}(m,l),\quad 1 \leq l \leq L,1 \leq m \leq \mu_k(L), 1\leq k \leq K, 
$$
and
$$
 s(e_{k,L,m}) = V_{k,L}(m,L), \quad 1 \leq m \leq \mu_k(L),1\leq k \leq K, 
$$
$$
 s(e_{k,0,m}) =  V(\eta_k), \quad 1 \leq m \leq \mu_k(0),1\leq k \leq K, 
$$
$$
t(e_{k,L,m}) =V_0,  \quad  1 \leq m \leq \mu_k(L),L \in \Bbb N,1 \leq  k \leq K,
$$
$$
s( {e^\top} ) = V(\ell), \quad t( {e^\top} ) = V(0). 
$$

The graph $G(K, \ell,  \omega)$ has a single subtree 
$G(\mathcal V(K, \ell,  \omega),  \mathcal F(K, \ell,  \omega)  )$.
One  has that
\begin{align*}
\card (\mathcal F(K, \ell,  \omega)) = \ell + \sum_{L\in \Bbb N,1 \leq k \leq K } L \mu_k(L).
\tag {4.2}
\end{align*}
The equality (4.1) holds in $G(K, \ell,  \omega)$ for $e = {e^\top} $ 
 and for  
 $e \in \mathcal E^{(K)}(K, \ell, \omega)$ and for no other edges.
We set
$$
M_k = \sum_{L \in \Bbb Z_+} \mu_k(L), \qquad 1 \leq k \leq K.
$$
One has that
\begin{align*}
\card(\mathcal E^{(k)}(K, \ell, \omega)) = M_k,  \qquad 1 \leq k \leq K, 
\end{align*}
and
\begin{align*}
\Delta^{(e)} = \sum_{1 \leq \kappa \leq k}M_k,  \qquad e \in \mathcal E^{(k)}(K, \ell, \omega)). \tag {4.3}
\end{align*}

We note special cases:
For $\ell = 0$ one obtains the single vertex graphs with more than one loop.
For $\ell \in \Bbb N, K= 1, \eta_1 = 0$ one obtains  bouquets of circles with the 
common point of the circles as the root of the  subtree. The subtree is a line graph precisely if $K = 1$ and $\eta_1 = \ell > 0$. Also note the non-empty intersection of the families $\bold F_I$ and $\bold F_{IV}$.

\begin{lemma}

Let $G=G(\mathcal V, \mathcal E)$ be a directed graph such that 
$G(\mathcal V, \mathcal F_G)$ is a tree. 
Let $G$ have an edge $ e\in\mathcal E  \setminus \mathcal F_G $ for which (4.1) holds.
Then there is an edge $ {e^\top}  \in \mathcal E  \setminus \mathcal F_G$ for which (4.1) holds, such that $s( {e^\top} ) $ is a leaf of the subtree.
\end{lemma}

\begin{proof}
In the case, that $s(e)$ is not an leaf of $G(\mathcal V, \mathcal F_G)$, there is a leaf 
$V$ of $G(\mathcal V, \mathcal F_G)$, that can be reached from the source vertex  of
$e$.  
The equality (4.1) implies, that $D(V)= 1$, and  that
(4.1) also holds for the outgoing edge of $V$.
\end{proof}

\begin{lemma}

For a graph $G(\mathcal V  , \mathcal E )$ such that $G (\mathcal V, \mathcal F_G)$ is a tree, 
and that has an edge $e\in\mathcal E \setminus \mathcal F$ such that
$$
 \beta(s(e)) = \lambda(s(e))+\card(  \mathcal E \setminus \mathcal F_G )
$$
there exist uniquely data 
$(K, \ell, (\eta_k, (\mu_k(L))_{L \in \Bbb Z_+})_{1 \leq k \leq K}) \in \Pi_{IV}$,
 such that $G(\mathcal V  , \mathcal E )$ is isomorphic to 
 $G(K, \ell, (\eta_k, (\mu_k(L))_{L \in \Bbb Z_+})_{1 \leq k \leq K})$.
\end{lemma}
\begin{proof}
By Lemma 4.1 we can
choose a edge $ {e^\top} \in \mathcal E \setminus \mathcal F_{G}$, such that 
$s( {e^\top} )$ is a leave of $G(\mathcal V, \mathcal F_G)$, 
such that (4.1) holds for $ {e^\top} $, and such that  
$\lambda(s( {e^\top} ))$ is maximal. 

The data $(K, \ell, (\eta_k, (\mu_k(L))_{L \in \Bbb Z_+})_{1 \leq k \leq K})$ are determined by $b(  {e^\top}  )$, and as $ {e^\top} $ satisfies (4.1), they do not depend on the choice of 
$ {e^\top} $. Let $\mathcal V(  {e^\top}  )$ denote the set of vertices, that are seen by 
$b(  {e^\top}  )$, and that have more than one outgoing edge.  
One has, that
$$
K =\card (\mathcal V( {e^\top} )), \quad
\ell = \lambda(   {e^\top} ).
$$
Set
$$
b(  {e^\top}  ) = (f(l))_{1\leq l \leq \ell},
$$
$$
V(l) = t(f(l) ), \quad 0 < l \leq \ell.
$$
One has, that
$$
\{\eta_k: 1 \leq k \leq K  \} = \{ \lambda (V): V \in \mathcal V(  {e^\top}  )  \}.
$$
Let $\mathcal E^{(k)}( {e^\top} )$ denote the set of $e\in \mathcal E \setminus \mathcal F$, such that $e \neq  {e^\top} $, and such that $b(e)$ traverses $V(\eta_k)$, or such that, in the case $k = 1, \eta_1 = 0$,  that $b(e)$ does not transverse any of the vertices $V(\eta_k  ), 1< k \leq K $.
One has, that
\begin{align*}
&\mu_k(0) = \card( \{   e \in    \mathcal E \setminus \mathcal F: s(e) = V(\eta_k)\}),
\\
&\mu_k(L) = \card(\{  e \in   \mathcal E^{(k)}(  {e^\top}  ) : \lambda(s(e))= \eta_k + L  \}), \quad 1\leq k \leq K,
L \in \Bbb N. \qed
\end{align*}
\renewcommand{\qedsymbol}{}
\end{proof}

\begin{theorem}
For directed graphs $G=G(\mathcal V  , \mathcal E )$, such that  
$\mathcal S(M\negthinspace {\scriptstyle D}({{G}}))$
 is  a Dyck inverse monoid,
  and such that there is an  
  $e\in  \mathcal M\negthinspace {\scriptstyle D}({{G}}) $,
   such that
\begin{align*}
\Delta^{(e)}=  
 \nu (M\negthinspace {\scriptstyle D}({{G}})) -1, \tag {4.4}
\end{align*}
the topological  conjugacy of the Markov-Dyck shifts 
$M\negthinspace {\scriptstyle D}({{G}})$ implies the isomorphism of the graphs 
$G$.
\end{theorem}
\begin{proof}

Consider a graph  $G=G(\mathcal V  , \mathcal E )$  such that (4.9) holds. 
The equalities (4.1) and (4.4) are equivalent. By Lemma (4.2) there exist data
$$
(K, \ell,  \omega) \in \Pi_{IV}, \quad
\omega =  (\eta_k, (\mu_k(L))_{L \in \Bbb Z_+})_{1 \leq k \leq K}),
$$
such that 
$M\negthinspace {\scriptstyle D}({G})$
  is isomorphic to
  $M\negthinspace {\scriptstyle D}({G(K, \ell, \omega)})$.
The parameters $K, \ell$ are recovered from topological conjugacy invariants of 
$M\negthinspace {\scriptstyle D}({G})$ by
$$
K = \card(\{ \Delta^{(e)}: e \in 
\mathcal M(M\negthinspace {\scriptstyle D}({{G}})   \}),
$$
$$
\ell = \max_{\{e \in \mathcal M(M\negthinspace {\scriptscriptstyle D}({{G}})): 
 \Delta^{(e)}=  
 \nu (M\negthinspace {\scriptscriptstyle D}({{G}})) -1\}}(\Lambda^{(e)} - 1),
$$
and as a consequence of (4.3) one obtains the parameters $M_k, 1 \leq k \leq K,$ by 
$$
\{ \sum_{1 \leq \kappa \leq k}M_\kappa : 1 \leq k \leq K\} = 
 \{ \Delta^{(e)}: e \in \mathcal M(M\negthinspace {\scriptstyle D}({{G}}))    \}.
$$

In the case that $K = 1$, one  obtains  $\omega =(\eta_1,(\mu_1(L))_{L \in \Bbb Z_+})$ from 
topological conjugacy invariants of $M\negthinspace {\scriptstyle D}({{G}})$, by
$$
\eta_1 = 
\frac{-I_2^{(0)}(M\negthinspace {\scriptstyle D}({{G}})) + 
 \sum_{e \in \mathcal M(M\negthinspace {\scriptscriptstyle D}({{G}}))}\Lambda^{(e)}}{M_1},
$$
which is  a consequence of (4.2),
and by
$$
\mu_1(L) = \card (\{ e \in \mathcal M( M\negthinspace {\scriptstyle D}({{G}} ):
\Lambda^{(e)}= \eta_1 + L + 1 \}), \quad 0 \leq L \leq \ell +1 -  \eta_1,
$$
which is  a consequence of (4.2).

In the case, that $K > 1$, we associate to the data 
$(K, \ell, \omega)$ a sequence 
\begin{align*}
(\widetilde{K}^{\langle k \rangle},\widetilde{ \ell}^{(k)},
\widetilde{ \omega}^{\langle k \rangle})
 \in \Pi_{IV}, \ \widetilde{ \omega}^{\langle k \rangle}  =  (\widetilde{\eta}^{\langle k \rangle}_\kappa, (\widetilde{\mu}^{\langle k \rangle}_\kappa(L))_{L \in \Bbb Z_+})_{1 \leq \kappa \leq k}),\ \ 1 \leq k \leq K,
\end{align*}
of auxiliary data by
\begin{align*}
&\widetilde{K}^{\langle k \rangle} = k, \\
&\widetilde{\ell}^{\langle k \rangle}= \ell, \\
&\widetilde{\eta}^{\langle k \rangle}_\kappa = \eta_\kappa - \eta_1, \quad   1\leq \kappa \leq k, \\
&\widetilde{\mu}^{\langle k \rangle}_\kappa(L)  = \begin{cases} 0, &\text{if  $L< \eta_1  $},  \\
\mu_k(L-\eta_1),&\text{if  $L\geq \eta_1  ,$} 
\end{cases}\qquad  1 \leq k \leq \kappa, \quad 1 \leq \kappa \leq K.
\end{align*}
Note the one-to-one correspondence between 
$\mathcal E^{(k)}(\widetilde{K}^{\langle k \rangle},\widetilde{ \ell}^{\langle k \rangle}   ,\widetilde{\omega}^{\langle k \rangle}) $ and
$\mathcal E^{(k)}(K, \ell, \omega),$ $
1 \leq k \leq K$.
We set 
$$
\delta_k = \eta_{k+1}- \eta_k,  \qquad  1\leq k < K.
$$
The parameters $ \delta_k,1\leq k < K, $ together with the sequence 
$(\widetilde{K}^{\langle k \rangle},\widetilde{ \ell}^{(k)},
\widetilde{ \omega}^{\langle k \rangle}),1 \leq k \leq K,$
of auxiliary data 
can be determined from invariants of topological conjugacy of  
$M\negthinspace {\scriptstyle D}({{G}})$
by an inductive procedure. The start of the induction is given by
$$
\widetilde{\mu}^{\langle 1 \rangle}_1(L) = \card(\{e \in \mathcal M( M\negthinspace {\scriptstyle D}({{G}})): 
\Delta^{(e)} = M_1, \Lambda^{(e)} = L + 1 \}) , \quad L \in \Bbb Z_+,
$$
and the induction step is given by
\begin{multline*}
\delta_{k} = \min\{d \in \Bbb N: 
 I^{(e)}_{\Lambda^{(e)}+2d}  - \Lambda^{(e)} > 
 \\
 I^{(\widetilde{e})}_{\Lambda^{(\widetilde{e})}+2d}  - \Lambda^{(\widetilde{e})},
  e \in \mathcal M( M\negthinspace {\scriptstyle D}({{G})),
  \Delta^{(e)} = \sum_{1 \leq \kappa \leq k}M_\kappa},
  \widetilde{e} \in\mathcal E^{(k)}(\widetilde{K}^{\langle k \rangle},\widetilde{ \ell}^{\langle k \rangle}   ,\widetilde{\omega}^{\langle k \rangle}) \},
  \\
  1 \leq k< K,
\end{multline*}
and
\begin{align*}
\widetilde{\eta}^{(k+1)}_\kappa = \widetilde{\eta}^{(k)}_\kappa, \
(\widetilde{\mu}^{(k+1)}_\kappa(L))_{L \in \Bbb Z_+} =(\widetilde{\mu}^{(k)}_\kappa(L))_{L \in \Bbb Z_+}, \quad 1\leq \kappa \leq k,  \quad 1 \leq k < K,
\end{align*}
and by
\begin{align*}
&\widetilde{\eta}^{(k+1)}_{k+1} =  \widetilde{\eta}^{(k)}_{k} + \delta_{k} ,   
\\
&\widetilde{\mu}^{(k+1)}_{k+1}(L) = \card(\{e \in \mathcal  E^{(k+1)}(M\negthinspace {\scriptstyle D}({{G}})) ) :
\Lambda ^{(e)} = L + 1 +  \widetilde{\eta}^{(k)}_k  \}), \ L \in \Bbb Z_+,
1 \leq k < K.
\end{align*}

It follows from (4.2), that
$$
\eta_1 = \frac{  - I_2^{(0)} (M\negthinspace {\scriptstyle D}({{G}})))+ \ell +1 +
\sum_{1 \leq k \leq K,L \in \Bbb Z_+}(L+1)\widetilde{\mu}^{(k)}_k  (L)
}
{\sum_{1 \leq k \leq K}
M_k}. 
$$
The graph $G(K, \ell, \omega  )$ is reconstructed from the topological conjugacy class of its Markov-Dyck shift by
$$
\eta_k = \eta_1 + \sum _{1 \leq \kappa  < k}\delta_\kappa, \ \
(\mu_k(L))_{L \in \Bbb Z_+} =(\widetilde{\mu}^{(k)}_k  (\eta_1 + L))_{L \in \Bbb Z_+},  \quad 1\leq k \leq K. \qed
$$
\renewcommand{\qedsymbol}{}
\end{proof}

	
\section{Spherically homogeneous directed graphs of height two}

We consider a directed graph   
 $G(\mathcal V  , \mathcal E )$ such that $G (\mathcal V, \mathcal F_G)$ is a tree, that has uniform height two. For spherically homogeneous graphs with a subtree of arbitrary  height see \cite{Kr5}.

\subsection{A criterion for spherical homogeneity for height two}

For an edge $e \in \mathcal E \setminus \mathcal F_G$ we denote by $\mathcal U^{(e)}_k$ the set of  cycles $a$ in $G( \mathcal V, \mathcal E^- \cup \mathcal E^+)$ at $V(0)$ of length $k \in \Bbb N$, such that the bi-infinite concatenation of $a$ yields a periodic point with multiplier
$\widehat{e}$. 
The out-degree of $V(0)$ we denote by $K$, and for  $e \in \mathcal E \setminus \mathcal F_G$ we denote
 the out-degree of $V^{(e)}_{1}$ by $L_e$, and the out-degree of $s(e)$ by $M_e$.

\begin{lemma}
\begin{align*}
I^{(e)}_5 = K + L_e + M_e, \qquad e \in \mathcal E \setminus \mathcal F_G. \tag {5.1}
\end{align*}
\end{lemma}
\begin{proof}
For all 
$e \in \mathcal E \setminus \mathcal F_G({{G}})$ 
every cycle in $\mathcal U^{(e)}_5$ is obtained by inserting into the cycle 
$f^-_1(e) f^-_2 (e)e$
either a loop 
$\widetilde f^-_1 \widetilde f^+_1, 
s(\widetilde f^-_1) =V(0)$, at $V(0)$, or a loop 
$\widetilde f^-_2 \widetilde f^+_2,
 s(\widetilde f^-_2) =V^{(e)}_{1}$, 
 at $V^{(e)}_{1}$, 
 or else a loop
$\widetilde e^-  \widetilde e^+, s(\widetilde e^-) =s(e^-)$ at $s(e^-).$
\end{proof}

\begin{lemma}
Let $I^{(e)}_5$ have the same value for all 
$e \in \mathcal E \setminus \mathcal F_G$.
Then
\begin{align*}
\card \thinspace (\{\widetilde f  \widetilde e:  \widetilde f \in  \mathcal F_G, s(\widetilde f) 
=
 V^{(e)}_{1},
 \widetilde e \in \mathcal E \setminus \mathcal F_G, s(\widetilde e) = t(\widetilde f ) \})
 = L_e M_e, \ \
 e \in \mathcal E \setminus \mathcal F_G.
\end{align*}
\end{lemma}
\begin{proof} 
By Lemma 5.1 
\begin{align*}
&\card \thinspace (\{\widetilde f  \widetilde e:  \widetilde f \in  \mathcal F_G, s(\widetilde f) 
= V^{(e)}_{1},
 \widetilde e \in \mathcal E \setminus \mathcal F_G, s(\widetilde e) = t(\widetilde f ) \})=\\
& \sum_{\{ \widetilde f\in \mathcal F_G: s(\widetilde f) = V^{(e)}_{1}  \}}  
\card \thinspace ( \{ \widetilde e \in \mathcal E \setminus \mathcal F_G,s(\widetilde e) =
 t(\widetilde f)   \}   ) = \\
& \sum_{\{ \widetilde f\in \mathcal F_G: s(\widetilde f) = V^{(e)}_{1} \}} M_{\widetilde e} =
 \sum_{\{ \widetilde f\in \mathcal F_G: s(\widetilde f) = V^{(e)}_{1}  \}}    
  (I_5 - L_{\widetilde e} - K     )   =             \\
& \sum_{\{ \widetilde f\in \mathcal F_G: s(\widetilde f) = V^{(e)}_{1}  \}} 
(I
 _5 - L_{ e} - K     )=
 \sum_{\{ \widetilde f\in \mathcal F_G: s(\widetilde f) = V^{(e)}_{1}  \}} M_e = L_eM_e. \qed
\end{align*}
\renewcommand{\qedsymbol}{}
\end{proof}

\begin{lemma}
\begin{multline*}
I^{(e)}_9 =   
K( K^2 + 3KL_e + 2KM_e + 7L_eM_e + 2L^2_e + 2M^2_e) +\\
 L_e(L^2_e + 3L_eM_e + 2M^2_e) + M_e^3,\qquad 
 e \in \mathcal E \setminus \mathcal F_G. \tag{5.2}
\end{multline*}
\end{lemma}
\begin{proof}
Count the cycles in  $\mathcal U^{(e)}_9 $ by applying Lemma 5.2.
\end{proof}

\begin{lemma}
\begin{align*}
I^{(e)}_{10} = K^2 + L^2_e + M^2_e + 3KL_e  + 3L_eM_e + 3KM_e , 
\qquad e \in \mathcal E \setminus \mathcal F_G. \tag {5.3}
\end{align*}
\end{lemma}
\begin{proof}
For all $e \in \mathcal E \setminus \mathcal F_G$ every cycle in $\mathcal U^{(e)}_{10}$ traverses the edge $e$ twice. Count these cycles by  applying Lemma 5.2.
\end{proof}

\begin{lemma}
Let  $I^{(e)}_5$,$I^{(e)}_9$,  and $I^{(e)}_{10}$ have the same value for all 
$ e \in \mathcal E \setminus \mathcal F_G.$ Then
\begin{align*}
L_e = 6I_5 + 4K -4 -
           \frac{I_5}{K}(I_5 +1)+
            \frac{1}{I_5}(1+ 4I_{10}-3K^2) +\frac{1}{KI_5} ( I_9  + I_{10}  ).  \tag {5.4}
\end{align*}
\end{lemma}
\begin{proof}
Insert (5.1), (5.2) and (5.3) into (5.4).
\end{proof}

\begin{theorem} 
Let $ I^{(e)}_5$,$I^{(e)}_9$,  and $I^{(e)}_{10}$  have the same value for all 
$ e \in \mathcal E \setminus \mathcal F_G$. Then $G$ is spherically homogeneous.
\end{theorem}
\begin{proof} 
The theorem follows from Lemma 5.1 and Lemma 5.5.
\end{proof}

\subsection{Spherically homogeneous directed graphs with a subtree of height two}

We denote the out-degree of the root of $\mathcal F_G$ by $K$, the out-degree of the vertices of $\mathcal F_G$ at level one  by $L$ and the out-degree of of the vertices of $\mathcal F_G$ at level two by $M$.
We suppress the Markov-Dyck shift $M\negthinspace {\scriptstyle D}({{G}})$ of $G$  in the notation, and set
$$
\tau = \card (\mathcal F_G).
$$
We note that

\begin{align*}
\nu = KLM, \tag {5.5}
\end{align*}

\begin{align*}
L = \frac{\tau - K}{K}, \tag {5.6}
\end{align*}

\begin{align*}
M = \frac{\nu}{\tau - K}, \tag {6.7}
\end{align*}

\begin{align*}
L = \frac{\nu}{M\tau - \nu}, \tag {5.8}
\end{align*}

\begin{align*}
K = \tau - \frac{\nu}{M}. \tag {5.9}
\end{align*}

\begin{lemma}
$M = 1$ if and only if
\begin{align*}
(\tau - \nu)(I_5  - 1) = (\tau - \nu)^2 + \nu. \tag {5.10}
\end{align*}
\end{lemma}
\begin{proof}
By lemma 5.1 and by (5.5) and (5.6) $M = 1$ implies (5.10).
Conversely, let (5.10) hold for $K, L, \widetilde M \in \Bbb N$. By Lemma 5.1 and by (5.5) and (5.6), (5.10) yields the equation 
\begin{align*}
L(1 + KL)\widetilde M^2 - (1 + KL + 2KL^2 + L - L^2)\widetilde M
+KL^2 +KL +1 - L^2 = 0. 
\end{align*}
with its  root $\widetilde M = 1$. For its other root $\widetilde M^\prime$ one finds that it cannot be a positive integer:
$$
\widetilde M^\prime = \frac{1 + KL + KL^2 - L^2}{L(1 + KL)} < \frac{1}{L} + 
\frac{L(K - 1)}{1 + KL} < \frac{1}{L}  + 1. \qed
$$
\renewcommand{\qedsymbol}{}
\end{proof}

\begin{lemma}
\begin{align*}
K^3 -(1 + \tau + I_5)K^2 + (2\tau - \nu + \tau I_5) K - \tau^2 =0. \tag {5.11}
\end{align*}
\end{lemma}
\begin{proof}
By Lemma 5.1 and by (5.6) and (5.7)
$$
I_5 = K +  \frac{\tau - K}{K}  + \frac{\nu}{\tau - K}  ,
$$ 
or
$$
K(\tau - K)I_5 = K^2(\tau - K) + (\tau - K)^2 + \nu K,
$$
which is (5.11).
\end{proof}

\begin{lemma}
\begin{align*}
I_4^0 =
K^2 - K + KL^2 + KLM^2 + K^2LM.
\tag {5.12}
\end{align*}
\end{lemma}
\begin{proof}
The Markov-Dyck shift of 
 $G$ has $K(K - 1) + KL + KL(L - 1)$ neutral periodic orbits of length four, that traverse only edges in $\mathcal F_G$, and it has
 $KLM(M - 1)$
neutral periodic orbits of length four, that traverse only edges in 
$\mathcal E \setminus \mathcal F_G$, and  it has $KLM + KLM(M - 1)$ other  neutral periodic orbits of length four.
\end{proof}

\begin{lemma}
\begin{align*}
K^3 + (\nu I_5  + \nu - 2\tau - I_4^{(0})K + \tau(\tau - \nu) = 0. \tag {5.13}
\end{align*}
\end{lemma}
\begin{proof}
To obtain (5.34) from (5.12), apply Lemma 5.1 and use (5.6). 
\end{proof}

We will use the notation
$$
a = 1 + \tau + I_5, \
b = (\nu - \tau)(I_5  + 2) - 2\tau - I_4^{(0)}, \
c = \tau(2\tau - \nu).
$$

\begin{lemma}
In the case that
$
M \geq 2,
$
one has 
that
\begin{align*}
K = \frac{1}{2a}(- b + \sqrt{b^2 - 4ac}), \tag {5.14}
\end{align*}
\begin{align*}
L = \frac{\tau - K}{K}, \tag {5.15}
\end{align*}
\begin{align*}
M = \frac{\nu}{\tau - K}. \tag {5.16}
\end{align*}
\end{lemma}
\begin{proof}
From Lemma 5.8 and Lemma 5.10 one has  the equation
\begin{align*}
aK^2 + bK + c = 0.\tag {5.17}
\end{align*}
If $c = 0$, then $K=-\frac{b}{a}$. If $c < 0$ then $K$ is equal to the positive root of  (5.17).
If $c > 0$, then $b < 0$, and
$
M < \frac{2}{L} + 2,
$
leaves the possibilities that $M = 2$, or that $M = 3$ and $L = 1$, and in both cases one confirms that
$
\frac{b}{a}> - 2K.
$
We have shown that (5.14) holds in all cases, and (5.15) and (5.16) hold by (5.6) and (5.7).
\end{proof}

\begin{theorem}
For spherically homogeneous directed graphs $G= G(\mathcal V, \mathcal E)$, if
$$
(\tau - \nu)(I_5  - 1) = (\tau - \nu)^2 + \nu,
$$ 
then
$$
K = \tau - \nu, \ L = \frac{\nu}{K},  \ M = 1,
$$
and if
$$
(\tau - \nu)(I_5  - 1)\not  = (\tau - \nu)^2 + \nu,
$$
then
$$
K = \frac{1}{2a}(- b + \sqrt{b^2 - 4ac}), \
L = \frac{\tau - K}{K}, \
M = \frac{\nu}{\tau - K}.
$$
\end{theorem}
\begin{proof}
See Lemma 5.7 and  Theorem 5.11  and also (5.8) and (5.9).
\end{proof}

\begin{corollary} 

(a) Spherical homogeneity is an invariant of topological conjugacy for the Markov-Dyck shifts of directed graphs with 
a  subtree of uniform depth two.

(b) For Markov-Dyck shifts of directed graphs with a  single subtree of uniform depth two, that are spherically homogeneous, topological conjugacy of the Markov-Dyck shifts implies isomorphism of the graphs.
\end{corollary} 
\begin{proof} 
See Theorem 5.6 and Theorem 5.11.
\end{proof}

\begin{corollary} 
Let $\widetilde G= G(\widetilde {\mathcal V},\widetilde {\mathcal E})$ and
$G=  G({\mathcal V}, {\mathcal E})$ be finite strongly connected directed graphs 
and let 
$$
\varphi: M\negthinspace {\scriptstyle D}({{\widetilde G}}) \to M\negthinspace {\scriptstyle D}({{G}}),
$$
be a topological conjugacy. Let $\mathcal S(\text{M{\Small D}}(G))$  be a Dyck inverse monoid.
Assume that
\begin{align*} 
\Lambda^{(e)} = 3, \qquad e \in \mathcal M(M\negthinspace {\scriptstyle D}({{G}})),\tag {$a$}
\end{align*} 
$(b)$ \qquad \ \
$I^{(e)}_5$,
$I^{(e)}_9$, and  \ $I^{(e)}_{10}$ 
have the same value for all 
$ e\in \mathcal M(M\negthinspace {\scriptstyle D}({{G}}))$,
\begin{align*}
(\tau - \nu)(I_5  - 1) = (\tau - \nu)^2 + \nu. \tag c
\end{align*}
Then there exists an automorphism $\beta$ of  $M\negthinspace {\scriptstyle D}({{G}})$ and an isomorphism
$$
\pi: \widetilde G \to G,
$$
such that one has for the topological conjugacy $\varphi_\pi$ of 
$M\negthinspace {\scriptstyle D}({{G}}))$ onto $M\negthinspace {\scriptstyle D}({{G}})$, that is induced by $\pi$, that
$$
\varphi = \beta \varphi_\pi.
$$
\end{corollary} 
\begin{proof} 
The hypothesis on the associated semigroup implies that $G$ and therefore also 
$\widetilde{G}$ 
has a single subtree.
Hypothesis $(a)$ implies that $G(\mathcal V, \mathcal F_G)$, and therefore also 
$G(\widetilde{\mathcal V}, \mathcal F_{\widetilde{G}})$, has uniformly depth two, and  hypothesis $(b)$  implies by Theorem 5.6 that 
$G$, and therefore also $\widetilde{G}$,  are spherically homogeneous. Hypothesis $(c)$ implies by Lemma 5.7  that
$\varphi$ induces a bijection
$$
\pi_\circ: \widetilde{\mathcal E} \setminus \mathcal F_{\widetilde{G}} \to  
 \mathcal E \setminus \mathcal F_G,
$$
which  can be extended  to an isomorphism $\pi: \widetilde G \to G$, by setting
$$
\pi(\widetilde V_0)  = V_0, \  \pi(\widetilde V^{(\widetilde e)}_{1})  = 
V^{(\pi_\circ(\widetilde e))}_{1},  \ \pi(s(\widetilde e)) = s( \pi_\circ (\widetilde e )),
$$
$$
\pi( \widetilde f^{(\widetilde e)}_0    ) = f^{( \pi_\circ (\widetilde e ))}_0, \ \pi(  \widetilde f^{(\widetilde e)}_1     ) =
 f^{(\pi_\circ (\widetilde e ))}_1.
$$ 
By construction
$$
\widehat \varphi_\pi = \widehat \varphi.
$$
Set
$$
\beta = \varphi \varphi_\pi^{-1}. \qed
$$
\renewcommand{\qedsymbol}{}
\end{proof}


\section{A family of three-vertex graphs}

We consider directed graphs $G = G(\mathcal V, \mathcal E)$, such that the semigroup
$\mathcal S(M\negthinspace {\scriptstyle D}({G}))$ 
is the graph inverse semigroup of a two-vertex graph $\widetilde{G}$, and such that the graph $G(\mathcal V ,\mathcal F_G)$ decomposes into a one-edge tree and a one-vertex  tree. 
We denote the source vertex of the edge of the one-edge tree by $\alpha_0$, and its target vertex  by 
$\alpha_1$, and
we denote the vertex of the one-vertex tree by
$\beta$.
For the adjacency matrix $A_G$ of the graph $G$ we choose the notation
\begin{multline*}
\left(
\begin{matrix}
 A_G(\alpha_0  ,\alpha_0   )  &     A_G(\alpha_0  ,  \alpha_1   )     & A_G(\alpha_0  ,  \beta   )
\\
    A_G ( \alpha_1 , \alpha_0  )     &     A_G( \alpha_1  ,  \alpha_1  )          & A_G(\alpha_1  ,  \beta   )
\\
      A_G(  \beta , \alpha_0  )    &          A_G( \beta   , \alpha_1   )     & A_G(  \beta  ,   \beta  )
\end{matrix}
\right)
=
\\
\left(
\begin{matrix}
 T_{\alpha \alpha} - \Delta^{(\alpha)}  &         1     & \Delta_\alpha
\\
    \Delta^{(\alpha)}      &     0          & T_{\alpha \beta}- \Delta_\alpha
\\
      T_{\beta\alpha }    &          0     & T_{\beta \beta}
\end{matrix}
\right), \tag {6.1}
\\
T_{\alpha \alpha}, T_{\beta \beta} \in \Bbb Z, \  T_{\alpha \beta}, T_{\alpha \beta} \in \Bbb N,
   \  0 \leq  \Delta_\alpha\leq  T_{\beta\alpha }  ,  0 \leq \Delta^{(\alpha)}  \leq T_{\beta \beta} .
\end{multline*}
It is required, that 
$$
T_{\alpha, \alpha} + T_{\beta, \alpha} > 1, \quad T_{\alpha, \beta} + T_{\beta, \beta} > 1.
$$
Also, if $\Delta^{(\alpha)} = 0$, then it is required, that $\Delta_\alpha < T_{\alpha, \beta} $,
and if $ T_{\alpha, \alpha} = T_{\beta, \beta} =  \Delta_\alpha =0  $, then it is required that 
$ T_{\alpha, \beta} \geq 2, T_{\beta, \alpha}\geq 2.$

Given a graph $G$ with adjacency matrix (6.1)   
set
\begin{align*}
&s =  T_{\alpha\alpha} +  T_{\beta\beta},
\\
&a = T_{\alpha \beta} + T_{\alpha\beta}, \tag {6.2}
\\
&b = T_{\alpha \beta} T_{\beta \alpha}, \tag {6.3}
\\
&c = \Delta _\alpha + T_{\alpha\beta}, \tag {6.4}
\\
&d = \Delta _\alpha T_{\alpha\beta}.\tag {6.5}
\end{align*}
One has that
\begin{align*}
T_{\beta \alpha} = \frac{b - d}{a - c}. \tag {6.6}
\end{align*}
As isomorphism invariants of the graph 
$\widehat{G}
$, the numbers $a$ and $b$, as well as the number $s$ 
are invariants of topological conjugacy. We note, that once also $c$ and $d$ are shown to be invariants of topological conjugacy the graph $G$ can be reconstructed from its Markov-Dyck shift by  (6.7) and  (6.2) or (6.3), and by (6.4) or (6.5). 
We also note that the Markov-Dyck shifts 
of graphs $G = G(\mathcal V, \mathcal E)$ such that 
$\widehat{G}  $ is a
 two-vertex graph, and such that $G(\mathcal V,\mathcal F_G)$ decomposes into a one-edge tree and a one-vertex tree,
are characterized by 
$$
I^{0}_2(M\negthinspace {\scriptstyle D}({G})) = 1 + s(M\negthinspace {\scriptstyle D}({G})) + 
a(M\negthinspace {\scriptstyle D}({G})).
$$
 
Given a graph $G$ with adjacency matrix (6.1)
we say that multipliers
$
\widehat{{e}}, 
\widehat{{\overline{{e}}}},
 \in \mathcal M_{1} (M\negthinspace {\scriptstyle D}({G}))$ are compatible, and write 
 $ \widehat{{e}} \sim \widehat{{\overline{{e}}}}$, 
if 
$\widehat{e}^-\widehat{\overline {e}}\thinspace ^- \in 
\mathcal S^-(M\negthinspace {\scriptstyle D}({G}))$. 
In terms of the graph $\widehat G$ the compatibility of 
$\widehat{{e}},\widehat{\overline {e}} \in \mathcal M_{1} (M\negthinspace {\scriptstyle D}({G}))$ means that $\widehat{{e}}$ and 
$\widehat{\overline {e}}$ are loops at the same vertex of $\widehat G$.
We denote by 
$\mathcal M_{1, 1} (M\negthinspace {\scriptstyle D}({G}))$
the set of multipliers of fixed points of $M\negthinspace {\scriptstyle D}({G}) $, and
we denote by 
$\mathcal M_{2, 1} (M\negthinspace {\scriptstyle D}({G}))$($\mathcal M_{2, 2} (M\negthinspace {\scriptstyle D}({G}))$) the set of multipliers of the orbits of length two of 
$M\negthinspace {\scriptstyle D}({G})$ of length one (two).

Consider  the set of graphs $G$ with adjacency matrix (6.1) such that
$$
\mathcal M_{2,1}( M\negthinspace {\scriptstyle D}({G}) )  \neq \emptyset,
$$
which is equivalent to the condition, that
$$
 \Delta^{(\alpha)} > 0.
$$ 
In this case the graph $G$ is reconstructed from its Markov-Dyck shift by
\begin{align*}
 \Delta^{(\alpha)} &= \card (\mathcal M_{2,1} (M\negthinspace {\scriptstyle D}({G}))  ),
  \\
T_{\alpha\alpha} -  \Delta^{(\alpha)} &= 
\card (\{\widehat{e} \in \mathcal M_{1, 1} (M\negthinspace {\scriptstyle D}({G}) ): 
 \widehat{e} \sim \widehat{\widetilde{e} }   \}    ), \quad 
\widehat{\widetilde{e}} \in \mathcal M_{2, 1} (M\negthinspace {\scriptstyle D}({G}) ),
\\
T_{\alpha\beta} &= I^{(\widehat{e})}_4(M\negthinspace {\scriptstyle D}({G}) ) -  T_{\alpha \alpha}
- 1 ,  \quad \widehat{e} \in \mathcal M_{2, 1} (M\negthinspace {\scriptstyle D}({G}) ),
\\
 \Delta_\alpha T_{\beta \alpha} &= I_2^- ( M\negthinspace {\scriptstyle D}({G}) ) - T_{\alpha\alpha}(T_{\alpha\alpha} - 1)
 - T_{\beta\beta}(T_{\beta\beta} - 1).
\end{align*}

We partition the set of graphs with adjacency matrix (6.1) such that $\Delta^{(\alpha)}= 0$
into three subsets.

\subsection{}
Consider the set of graphs $G$ with adjacency matrix (6.1), such that
$$
\mathcal M_{2,1}(M\negthinspace {\scriptstyle D}({G})) = \emptyset,
$$
and such that there are
$\widehat{e},\widehat{ \widetilde{e}} \in \mathcal M_{1, 1} (M\negthinspace {\scriptstyle D}({G}))$ that are incompatible, which is equivalent to the condition that
$$
 \Delta^{(\alpha)} = 0, \qquad  T_{\alpha\alpha} > 0,  \  T_{\beta\beta} > 0.
$$
With a choice of $\widehat{e}, \widehat{\widetilde{e}}\in \mathcal M_{1, 1} (M\negthinspace {\scriptstyle D}({G})), \widehat{e} \not \sim \widehat{\widetilde{e}}$, 
set
$$
T_{\widehat{e}} = \card ( \{ \widehat{e}^\prime \in \mathcal M_{1, 1} (M\negthinspace {\scriptstyle D}({G})) : \widehat{e}^\prime \sim \widehat{e} \} ), 
\ \
T_{\widehat{\widetilde{e}}} = \card ( \{ \widehat{e}^\prime \in \mathcal M_{1, 1} (M\negthinspace {\scriptstyle D}({G})) : \widehat{e}^\prime \sim \widehat{\widetilde{e}} \} ),
$$
The reconstruction of the graph $G$ from its Markov-Dyck shift is by
\begin{align*}
I_4^{(\widehat{e})} (M\negthinspace {\scriptstyle D}({G}) )  + I_4^{(\widehat{\widetilde{e}})} (M\negthinspace {\scriptstyle D}({G})) 
-T_{\widehat{e}} -T_{\widehat{\widetilde{e}}}  - 1 
&= \Delta_\alpha + T_{\beta\alpha} = c( M\negthinspace {\scriptstyle D}({G}) ),
\\
I_2^-(M\negthinspace {\scriptstyle D}({G}) )  - T_{\widehat{e}}(T_{\widehat{e}} - 1) -T_{\widehat{\widetilde{e}}}(T_{\widehat{\widetilde{e}}} - 1) 
&= \Delta_\alpha T_{\beta\alpha} = d(M\negthinspace {\scriptstyle D}({G}) ).
\end{align*}

\subsection{}
Consider the set of graphs $G$ with adjacency matrix (6.1), such that 
$$
\mathcal M_{2,1}(M\negthinspace {\scriptstyle D}({G})  = \emptyset,  \quad I^-_1(M\negthinspace {\scriptstyle D}({G}) ) > 0,
$$
and such that all $\widehat{e} \in  \mathcal M_{1,1}( M\negthinspace {\scriptstyle D}({G} )  ) $ are compatible, which is equivalent to the condition, that
$$
 \Delta^{(\alpha)} = 0, \qquad  
 T_{\alpha\alpha} + T_{\beta\beta} > 0, \ T_{\alpha\alpha}  T_{\beta\beta}  = 0.
$$
Under these assumptions  one has that
\begin{multline*}
- 1 - I^-_1(M\negthinspace {\scriptstyle D}({G}) ) +
\\
\tfrac{1}{2}(  \min \{ I_4^{(\mu)}:\mu \in
 \mathcal M_2(M\negthinspace {\scriptstyle D}({G}) ) \} +
 \max \{ I_4^{(\mu)} :\mu \in \mathcal M_2(M\negthinspace {\scriptstyle D}({G}) ) \} =
 \\
  \Delta_\alpha + T_{\beta\alpha} = c(M\negthinspace {\scriptstyle D}({G}) ),
 \end{multline*}
 \begin{align*}
 I_2ã(M\negthinspace {\scriptstyle D}({G}) ) - I_1^-( M\negthinspace {\scriptstyle D}({G})  )(I_1^-( M\negthinspace {\scriptstyle D}({G})  )   - 1) = \Delta_\alpha  T_{\beta\alpha} = d(M\negthinspace {\scriptstyle D}({G})  ).
\end{align*}
It follows, that 
\begin{align*}
A_G =
\left(
\begin{matrix}
 T_{\alpha \alpha}   &         1     & \Delta_\alpha
\\
   0     &     0          & T_{\alpha \beta}- \Delta_\alpha
\\
      T_{\beta\alpha }    &          0     & 0
\end{matrix}
\right), \tag {6.7}
\end{align*}
or that
\begin{align*}
A_G =
\left(
\begin{matrix}
 0  &         1     & \Delta_\alpha
\\
   0     &     0          & T_{\alpha \beta}- \Delta_\alpha
\\
      T_{\beta\alpha }    &          0     & T_{\beta \beta} 
\end{matrix}
\right). \tag {6.8}
\end{align*}
One distinguishes three cases.

\noindent
$\bold{6.2.a.}$ Assume moreover, that
$$
1 + \Delta_\alpha \neq T_{\beta\alpha }.
$$
Under this additional assumption, if 
$$
I_4^{(\widehat{e})} =   I_2^-(M\negthinspace {\scriptstyle D}({G}) ) + 1 + \Delta_\alpha, \quad \widehat{e} \in \mathcal M_{1,1} 
 (M\negthinspace {\scriptstyle D}({G}) ),
$$
then $A_G$ is given by (6.7), and if
$$
I_4^{(\widehat{e})} = I_2^-(M\negthinspace {\scriptstyle D}({G})) + T_{\beta, \alpha}, \quad 
\widehat{e} \in \mathcal M_{1,1}  
(M\negthinspace {\scriptstyle D}({G}) ),
$$
then $A_G$ is given by (6.7).

\noindent
$\bold{6.2.b.}$ Assume moreover that
$$
1 + \Delta_\alpha = T_{\beta\alpha }, \qquad  T_{\alpha\beta } -  \Delta_\alpha \neq 
T_{\beta\alpha }   .
$$
Under this additional assumption, if 
$$
I_6^{(\widehat{e})} (M\negthinspace {\scriptstyle D}({G}) )= T_{\alpha\beta } -  \Delta_\alpha, \quad 
\widehat{e} \in \mathcal M_{1,1}  
(M\negthinspace {\scriptstyle D}({G}) ),
$$
then $A_G$ is given by (6.7), and if
$$
I_6^{(\widehat{e})}(M\negthinspace {\scriptstyle D}({G})) = T_{\beta\alpha },
\quad 
\widehat{e} \in \mathcal M_{1,1}  
(M\negthinspace {\scriptstyle D}({G}) ),
$$
 then $A_G$ is given by (6.8).

\noindent
$\bold{6.2.c.}$ Assume moreover, that
\begin{align*}
1 + \Delta_\alpha = T_{\beta\alpha } = T_{\alpha\beta } -  \Delta_\alpha. \tag {6.9}
\end{align*}
One distinguishes two cases.

\noindent
$\bold{6.2.c.I.}$ Assume further, that
$$
 \Delta_\alpha = 0.
$$
Under this further assumption one has that
\begin{align*}
A_G =
\left(
\begin{matrix}
 I_1^-( M\negthinspace {\scriptstyle D}({G}))  &         1     & 0
\\
   0     &     0          & 1
\\
      1   &          0     & 0
\end{matrix}
\right), 
\end{align*}
or that
\begin{align*}
A_G =
\left(
\begin{matrix}
 0  &         1     & 0
\\
   0     &     0          & 1
\\
      1    &          0     & I_1^-(M\negthinspace {\scriptstyle D}({G}) ) 
\end{matrix}
\right).
\end{align*}
These two adjacency matrices yield isomorphic graphs.

\noindent
$\bold{6.2.c.II.}$ Assume further, that
\begin{align*}
\Delta_\alpha > 0.
\end{align*}
Under this further assumption one has from (6.9) that
$$
T_{\alpha\beta } = 1 + \Delta_\alpha, \quad T_{\beta\alpha } = 1 + 2\Delta_\alpha.
$$
If $\widehat{G}
$ is isomorphic to the graph with adjacency matrix
\begin{align*}
\left(
\begin{matrix}
 I_1^-(M\negthinspace {\scriptstyle D}({G}) ) &         1 + \Delta_\alpha    
\\
    1 + 2\Delta_\alpha    &     0   
 \end{matrix}
\right),
\end{align*}
then 
\begin{align*}
A_G =
\left(
\begin{matrix}
 I_1^-( M\negthinspace {\scriptstyle D}({G}))  &         1     & \Delta_\alpha
\\
   0     &     0          & 1
\\
      1 + 2\Delta_\alpha  &          0     & 0
\end{matrix}
\right), 
\end{align*}
and if $\widehat{G}
$ is isomorphic to the graph with adjacency matrix
\begin{align*}
\left(
\begin{matrix}
 I_1^-(M\negthinspace {\scriptstyle D}({G})) &         1 +2 \Delta_\alpha    
\\
    1 + \Delta_\alpha    &     0   
 \end{matrix}
\right),
\end{align*}
then
\begin{align*}
A_G =
\left(
\begin{matrix}
 0  &         1     &  \Delta_\alpha
\\
   0     &     0          & 1
\\
      1 + 2\Delta_\alpha    &          0     & I_1^-(M\negthinspace {\scriptstyle D}({G}) ) 
\end{matrix}
\right).
\end{align*}

\subsection{}
Consider the set of graphs $G$ with adjacency matrix (6.1), such that  one has that
$$
I_1^-(M\negthinspace {\scriptstyle D}({G}) ) = 0,
$$
which is equivalent to the condition, that
$$
T_{\alpha\alpha } = T_{\beta\beta } = 0.
$$ 
One distinguishes two cases.

\noindent
$\bold{6.3.a.}$ Assume, that also
$$
I_2^-( M\negthinspace {\scriptstyle D}({G}) ) > 0,
$$
which is equivalent to the condition, that
$$
\Delta_\alpha > 0.
$$
Under this further assumption the reconstruction if the graph $G$ from 
$M\negthinspace {\scriptstyle D}({G})$ is by
$$
-1 + I_4^{(\mu)}( M\negthinspace {\scriptstyle D}({G}) )  = \Delta_\alpha + T_{\alpha\beta} = c(M\negthinspace {\scriptstyle D}({G}) ), \quad \mu \in
\mathcal M_{2,2}(M\negthinspace {\scriptstyle D}({G}) ),
$$
$$
I_2^-(M\negthinspace {\scriptstyle D}({G}) ) =  \Delta_\alpha T_{\beta\alpha} =  d(M\negthinspace {\scriptstyle D}({G})).
$$
\noindent
$\bold{6.3.b.}$ Assume, that also
$$
I_2^-(M\negthinspace {\scriptstyle D}({G}))= 0,
$$
which is equivalent to the condition, that
$$
\Delta_\alpha = 0.
$$

Denote by $I_2^-(\alpha)$ the cardinality of the set of points 
$(p_i)_{i \in \Bbb Z}$
of period two of $M\negthinspace {\scriptstyle D}({G})$, such that
$
s(p_0) \in \{ \alpha_0,  \alpha_1  \},   
$
and by 
$I_2^-(\beta)$ the cardinality of the set of points 
$(p_i)_{i \in \Bbb Z}$
of period two of $M\negthinspace {\scriptstyle D}({G})$, such that
$
s(p_0) = \beta.
$
By \cite[Corollary 2.3] {KM} the set
$
\{ I_2^0(\alpha), I_2^0(\beta) \}
$
is an invariant of topological conjugacy. One has, that
$$
I_2^0(\alpha) = 1 +  T_{\alpha\beta}  , \quad I_2^0(\beta) = T_{\beta\alpha}  .
$$
The graph $G$ is reconstructed  by
$$
\{ T_{\beta\alpha}  \} = \{ T_{\alpha\beta}, T_{\beta\alpha} \}\cap
\{ I_2^0(\alpha) , I_2^0(\beta) \}.
$$

\medskip
\noindent
In summary we state a theorem.

\begin{theorem}
For graphs $G = G(\mathcal V, \mathcal E)$ such that the semigroup
$\mathcal S(M\negthinspace {\scriptstyle D}({G}))$ 
is the graph inverse semigroup of a two-vertex graph, and such that the graph 
$G(\mathcal V, \mathcal  F_G)$ decomposes into a one-edge tree and a one-vertex tree, the topological conjugacy of their Markov-Dyck shift implies the isomorphism of the graphs.
\end{theorem}

\bigskip

\par\noindent Toshihiro Hamachi
\par\noindent Faculty of Mathematics
\par\noindent Kyushu University
\par\noindent 744 Motooka, Nishi-ku
\par\noindent Fukuoka 819-0395, 
\par\noindent Japan
\par\noindent t.hamachi.796@m.kyushu-u.ac.jp

\bigskip

\par\noindent Wolfgang Krieger
\par\noindent Institute for Applied Mathematics, 
\par\noindent  University of Heidelberg,
\par\noindent Im Neuenheimer Feld 205, 
 \par\noindent 69120 Heidelberg,
 \par\noindent Germany
\par\noindent krieger@math.uni-heidelberg.de

\bigskip


\begin{thebibliography}{9999}

\bibitem[HI] {HI}
{\sc T.~Hamachi, K.~Inoue},
{\it Embeddings of shifts of finite type into the Dyck shift},
Monatsh. Math.
{\bf  145}
(2005),
107 -- 129.             

\bibitem[HIK] {HIK}
{\sc T.~Hamachi, K.~Inoue, W.~Krieger},
{\it Subsystems of finite type and semigroup invariants of subshifts},
J. reine angew. Math.
{\bf  632}
(2009),
37 -- 61.         

 
 \bibitem[HK] {HK}
{\sc T.~Hamachi and W.~Krieger},
{\it A construction of subshifts and a class of semigroups},
 arXiv: 1303.4158 [math.DS] 

 
\bibitem[Ki]{Ki}{\sc B.~P.~Kitchens},
{\it Symbolic dynamics}, Springer, Berlin, Heidelberg, New York
(1998)

\bibitem[Kr1]{Kr1}
{\sc  W.~Krieger},
{\it  On the uniqueness of the equilibrium state},
 Math. Systems Theory
  {\bf  8}
(1974)
97 -- 104

\bibitem[Kr2]{Kr2}
{\sc  W.~Krieger},
{\it  On a syntactically defined invariant of symbolic dynamics},
Ergod. Th. \& Dynam. Sys.
  {\bf  20}
(2000),
501 -- 516

\bibitem[Kr3]{Kr3}
{\sc  W.~Krieger},
{\it  On subshifts and  semigroups},
Bull. London Math. Soc.
{\bf  38}
(2006),
617 -- 624


\bibitem[Kr4]{Kr4}
{\sc  W.~Krieger},
{\it  On subshift presentations},
Ergod. Th. \& Dynam. Sys.
{\bf  37}
(2017),
1253 -- 1290

  
 \bibitem[Kr5]{Kr5}
{\sc  W.~Krieger},
{\it  On a class of highly symmetric Markov-Dyck shifts},
arXiv:1806.02303 [math.DS].
 
\bibitem[KM]{KM}
{\sc  W.~Krieger, K.~Matsumoto},
{\it  Markov-Dyck shifts, neutral periodic points and topological conjugacy},
Discrete and Continuous Dynamical systems
 {\bf 39} 
(2019),
1 -- 18

 
 \bibitem[L]{L}
 {\sc M.~V.~Lawson},
{\it Inverse semigroups},
 World Scientific, Singapur, New Jersey, London and Hong Kong
(1998).

\bibitem[LM]{LM}{\sc D.~Lind and B.~Marcus},
{\it An introduction to symbolic dynamics and coding},
 Cambridge University Press, Cambridge
(1995)

 
 
 \bibitem [NP]{NP}{\sc M. ~Nivat and J.-F. ~Perrot},
{\it  Une g\'{e}n\'{e}ralisation du mono\^{i}de bicyclique},
 C. R. Acad. Sc. Paris, 
 {\bf 271}
(1970), 
824--827.

\end{thebibliography}
 \end{document}